\documentclass{svmult}
\usepackage{type1cm}  
\usepackage{amsmath,amssymb,amsfonts,bbm}
\usepackage[varg,cmintegrals,cmbraces]{newtxmath}
\usepackage{newtxtext}
\usepackage[cal=euler]{mathalfa}

\usepackage{multicol,makeidx} 
\usepackage[bottom]{footmisc}   

\makeindex 

\usepackage{tikz-cd}
\tikzcdset{
arrow style=tikz,
diagrams={>={Straight Barb[scale=0.714]}}
}
\PassOptionsToPackage{normalem}{ulem}
\usepackage{ulem}
\usepackage{url}
\makeatletter

\author{Denis-Charles Cisinski}
\title*{Cohomological Methods in Intersection Theory}
\institute{D.-C. Cisinski \at Fakult\"at f\"ur Mathematik,
Universit\"at Regensburg, 93040 Regensburg, Germany\\ \email{denis-charles.cisinski@ur.de}}
\spnewtheorem{thm}{\protect\theoremname}[subsection]{\bfseries}{\itshape}
\spnewtheorem{prop}[thm]{\protect\propositionname}{\bfseries}{\itshape}
\spnewtheorem{Lemma}[thm]{\protect\lemmaname}{\bfseries}{\itshape}
\spnewtheorem{cor}[thm]{\protect\corollaryname}{\bfseries}{\itshape}

\spnewtheorem{rem}[thm]{\protect\remarkname}{\itshape}{}
\spnewtheorem{notation}[thm]{Notations}{\itshape}{}
\spnewtheorem{Example}[thm]{\protect\examplename}{\itshape}{}
\spnewtheorem{xca}[thm]{\protect\exercisename}{\itshape}{}

\spnewtheorem{defn}[thm]{\protect\definitionname}{\bfseries}{}
\spnewtheorem{num}[thm]{\nocaption}{\bfseries}{}
\spnewtheorem{paragr}[thm]{\nocaption}{\bfseries}{}
\spnewtheorem{assumption}[thm]{\nocaption}{\bfseries}{}

\@ifundefined{date}{}{\date{}}
\makeatother

  \providecommand{\corollaryname}{Corollary}
  \providecommand{\definitionname}{Definition}
  \providecommand{\examplename}{Example}
  \providecommand{\exercisename}{Exercise}
  \providecommand{\propositionname}{Proposition}
  \providecommand{\remarkname}{Remark}
  \providecommand{\theoremname}{Theorem}
  \providecommand{\lemmaname}{Lemma}

\global\long\def\colim{\operatorname*{colim}}
\global\long\def\lim{\operatorname*{lim}}

\newcommand{\ZZ}{\mathbf{Z}}
\newcommand{\QQ}{\mathbf{Q}}

\newcommand{\Hom}{\mathrm{Hom}}
\newcommand{\uHom}{\mathit{Hom}}
\newcommand{\sHom}{\uHom}

\newcommand{\spec}[1]{\mathrm{Spec}(#1)}

\newcommand{\DM}{\mathit{DM}}
\newcommand{\h}{h}
\newcommand{\DMh}{\DM_\h}
\newcommand{\DMhlc}{\DM_{\h,lc}}
\newcommand{\uDM}{\underline{\DM}}
\newcommand{\eff}{\mathit{eff}}
\newcommand{\corr}{\mathit{MCorr}}

\begin{document}
\maketitle
\abstract{These notes are an account of a series of lectures I gave at the
LMS-CMI Research School `Homotopy Theory and Arithmetic Geometry: Motivic and Diophantine Aspects',
in July 2018, at the Imperial College London. The goal of these notes is to see how
motives may be used to enhance cohomological methods, giving natural ways to
prove independence of $\ell$ results and constructions of characteristic classes (as $0$-cycles). This leads to the Grothendieck-Lefschetz formula,
of which we give a new motivic proof.
There are also a few additions to what have been told in the lectures:
\begin{itemize}
\item A proof of Grothendieck-Verdier duality of \'etale motives on schemes of finite
type over a regular quasi-excellent scheme (which slightly improves
the level of generality in the existing literature).
\item A proof that $\QQ$-linear motivic sheaves are virtually
integral (Theorem \ref{virtual}). 
\item A proof of the motivic generic base change formula. 
\end{itemize}
I am grateful to Shachar Carmeli for having allowed me to
use the notes he typed from my lectures,
and to K\'evin Fran\c cois for finding a gap
in the proof of the motivic generic base formula. While
preparing these lectures and writing these notes,
I was partially supported by the SFB 1085 ``Higher Invariants'' funded by the Deutsche Forschungsgemeinschaft (DFG).}
\abstract*{These notes are an account of a series of lectures I gave at the
LMS-CMI Research School `Homotopy Theory and Arithmetic Geometry: Motivic and Diophantine Aspects',
in July 2018, at the Imperial College London. The goal of these notes is to see how
motives may be used to enhance cohomological methods, giving natural ways to
prove independence of $\ell$ results and constructions of characteristic classes (as $0$-cycles). This leads to the Grothendieck-Lefschetz formula,
of which we give a new motivic proof.
There are also a few additions to what have been told in the lectures:
\begin{itemize}
\item A proof of Grothendieck-Verdier duality of \'etale motives on schemes of finite
type over a regular quasi-excellent scheme (which slightly improves
the level of generality in the existing literature).
\item A proof that $\QQ$-linear motivic sheaves are virtually
integral (Theorem \ref{virtual}). 
\item A proof of the motivic generic base change formula. 
\end{itemize}
I am grateful to Shachar Carmeli for having allowed me to
use the notes he typed from my lectures,
and to K\'evin Fran\c cois for finding a gap
in the proof of the motivic generic base formula. While
preparing these lectures and writing these notes,
I was partially supported by the SFB 1085 ``Higher Invariants'' funded by the Deutsche Forschungsgemeinschaft (DFG).}

\section*{Introduction}
Let $p$ be a prime number and $q=p^r$ a power of $p$.
Let $X_{0}$ be a smooth and projective algebraic variety
over $\mathbf{F}_{q}$. It comes equipped with the geometric Frobenius map
$\phi_r:X\to X$,
where $X=X_{0}\times_{\mathbf{F}_{q}}\bar{\mathbf{F}}_{p}$,
so that the locus of fixed points of
$F$ corresponds to the set of rational points of $X_0$
(various Frobenius morphisms
and their actions are discussed in detail
in Remark \ref{rem:construction of Frobenius actions} below).
We may take the graph of Frobenius $\Gamma_{\phi_r}\subset X\times X$,
intersect with the diagonal, then interpret the intersection number
cohomologically with the formula of Lefschetz
through $\ell$-adic cohomology, with $\ell\neq p$.

For each $Z\subseteq X$ we can attach a cycle $[Z]\in H^{*}(X,\QQ_{\ell})$
and do intersection theory (interpreting geometrically the algebraic operations on cycle classes). 
For instance, if $Z'\subseteq X$ is another cycle which is transversal to $Z$, we have
$[Z]\cdot [Z']=[Z\cap Z']$. Together with Poincar\'e duality, this implies that the
number of rational points of $X_0$ may be computed cohomologically:
\[ \# X(\mathbf{F}_q)=\sum_i (-1)^i\, \mathrm{Tr}\big(\phi_r^*:H^i(X,\QQ_\ell)\to H^i(X,\QQ_\ell)\big)\, .\]
The construction of $\ell$-adic cohomology by Grothendieck
was aimed precisely at proving this kind of formulas, with the goal of proving Weil's conjectures
on the $\zeta$-functions of smooth and projective varieties over finite fields, which was finally achieved by Deligne~\cite{Weil1,Weil2}.

Here are two natural problems we would like to discuss: 
\begin{itemize}
\item Extend this to non-smooth or non-proper schemes and to
cohomology with possibly non-constant coefficients:
this is what the Grothendieck-Lefschetz formula
is about.
\item Address the problem of independance on $\ell$
(when we compute traces of endomorphisms with a less obvious
geometric meaning): this is what motives are made for.
\end{itemize}
In this series of lectures,
I will explain what is a motive and explain how to prove a motivic Grothendieck-Lefschetz formula. To be more precise, we shall work with \emph{$\h$-motives} over a scheme $X$,
which are one of the many descriptions of \'etale motives. These are the objects of the
triangulated category $\DMh(X)$ constructed and studied in details in \cite{CD4},
which is a natural modification (the non-effective version)
of an earlier construction of Voevodsky \cite{voe0},
following the lead of Morel and Voevodsky into the realm of
$\mathbf{P}^1$-stable $\mathbf{A}^1$-homotopy theory of schemes.
Although we will not mention them in these notes, we should mention that
there are other equivalent constructions of \'etale motives which
are discussed in \cite{CD4} and \cite{ay3} (not to speak of the
many models with $\QQ$-coefficients discussed in \cite{CD3}),
and more importantly, that
there are also other flavours
of motives \cite{FSV,kelly,CD5}, which are closer to geometry
(and further from topology), for which one can still prove
Lefschetz-Verdier formulas; see \cite{trace}.
As we will see later, \'etale motives with
torsion coefficients may be identified with classical \'etale sheaves.
In particular, when restricted to the case of torsion coefficients, all the
results discussed in these notes on trace formulas go back to
Grothendieck \cite{SGA5}. The case of rational coefficients has also
been studied previously to some extend by Olsson \cite{Olsson1,Olsson2}.
We will see here how
these fit together, as statements about \'etale motives with arbitrary
(e.g. integral) coefficients.
Finally, we will recall the Lefschetz-Verdier trace formula
and explain how to deduce from it
the motivic Grothendieck-Lefschetz formula,
using Bondarko's theory of weights and Olsson's computations of
local terms of the motivic Lefschetz-Verdier trace.

\section{\'Etale motives }

\subsection{The $\h$-topology}
\begin{defn}
A morphism of schemes $f:X\to Y$ is a \emph{universal topological isomorphism}
\index{universal!topological homeomorphism}
(\emph{epimorphism}\index{universal!topological epimorphism} resp.)
if for any map of schemes $Y'\to Y$, the pullback
$X'=Y'\times_{Y}X\to Y'$ is a homeomorphism (a topological epimorphism resp.,
which means that it is surjective and exhibits $Y'$ as a topological quotient).
\end{defn}
\begin{Example}
Surjective proper maps as well as faithfully flat maps all are universal epimorphisms.
\end{Example}
\begin{prop}
A morphism of schemes $f:X\to Y$ is a universal homeomorphism if
and only if it is surjective radicial and integral. Namely, $f$ is integral
and, for any algebraically closed field $K$, induces a bijection
$X(K)\cong Y(K)$.
\end{prop}
\begin{Example}
The map $X_{red}\to X$ is a universal homeomorphism. 
\end{Example}
\begin{Example}
Let $K'/K$ be a purely inseparable extension of fields.
If $X$ is a normal scheme with field of functions $K$, and if $X'$ is the
normalization of $X$ in $K'$, then the induced map $X'\to X$ is a universal homeomorphism.
\end{Example}

Following Voevodsky~\cite{voe0}, we can define the
\emph{$\h$-topology}\index{h-topology@$\h$-topology}
as the Grothendieck topology
on noetherian schemes associated to the pre-topology whose coverings are finite families
$\{X_{i}\to X\}_{i\in I}$ such that the induced map
$\coprod_{i}X_{i}\to X$ is a universal epimorphism.\footnote{As
shown by D.~Rydh~\cite{Rydh},
this topology can be extended to all schemes, at the price of adding compatiblities
with the constructible topology.} Beware that the $\h$-topology is not subcanonical: any universal
homeomorphism becomes invertible locally for the $\h$-topology.

Using Raynaud-Gruson's flatification theorem, one shows
the following; see~\cite{Rydh}.
\begin{thm}
(Voevodsky, Rydh): Let $X_{i}\to X$ be an $h$-covering. Then there
exists an open Zariski cover $X=\cup_{j}X_{j}$ and for each $j$
a blow-up $U_{j}'=Bl_{Z_{j}}U_{j}$ for some closed subset $Z_{j}\subseteq U_{j}$,
a finite faithfully flat $U_{j}''\to U_{j}'$ and a Zariski covering
$\{V_{j,\alpha}\}_{\alpha}$ of $U_{j}''$ such that we have a dotted arrow making
the following diagram commutative. 

\[
\begin{tikzcd}
\coprod_{j,\alpha}V_{j,\alpha}\ar{d}\ar[dotted]{rrr} & & & \coprod_{i}X_{i}\ar{d}\\
\coprod_{j}U_{j}''\ar{r} &
\coprod_{j}U_{j}'\ar{r} &  
\coprod_{j}U_{j}\ar{r} & X
\end{tikzcd}
\]
\end{thm}
This means that the property of descent with respect to the $\h$-topology is exactly the property
of descent for the the Zariski topology, together with proper descent.
\begin{rem}
Although Grothendieck topologies where not invented yet,
a significant amount of the results of SGA~1 \cite{SGA1}
are about $h$-descent of \'etale sheaves (and this is one of the reasons
why the very notion of descent was introduced in SGA~1).
This goes on in SGA~4 \cite{SGA4} where the fact that proper surjective
maps and \'etale surjective maps are morphism of universal cohomological
descent is discussed at length. However, it is only in Voevodsky's thesis
\cite{voe0} that the $h$-topology is defined and studied properly, with
the clear goal to use it in the definition of a triangulated
category of \'etale motives.
\end{rem}
\subsection{Construction of motives, after Voevodsky}
\begin{paragr}
Let $\Lambda$ be a commutative ring. Let $Sh_{\h}(X,\Lambda)$ denote
the category of sheaves of $\Lambda$-modules on the category of separated
schemes of finite type over $X$ with respect to the $\h$-topology. We have
Yoneda functor
\[
Y\mapsto\Lambda(Y)\, ,
\]
where $\Lambda(Y)$ is the $\h$-sheaf associated to the presheaf $\Lambda[\Hom_{X}(-,Y)]$
(the free $\Lambda$ module generated by $\Hom_{X}(-,Y)$).

Let us consider the derived category $D(Sh_{\h}(X,\Lambda))$, i.e. the localization
of complexes of sheaves by the quasi-isomorphisms.
Here we will speak the language of $\infty$-categories
right away.\footnote{We refer to \cite{HTT,HA} in general.
However, most of the literature on motives is written using
the theory of Quillen model structures. The precise way to
translate this language to the one of $\infty$-categories
is discussed in Chapter 7 of \cite{hcha}.}
In particular, the word `localization' has to be interpreted
higher categorically (if we take as models simplicial categories, this is
also known as the Dwyer-Kan localization). 
That means that $D(Sh_{\h}(X,\Lambda))$ is in fact a stable $\infty$-category
with small limits and colimits (as is any localization of a stable
model category). Moreover, the constant sheaf functor turns it into
an $\infty$-category enriched in the monoidal stable $\infty$-category
$D(\Lambda)$ of complexes of $\Lambda$-modules (i.e. the localization of the
category of chain complexes of $\Lambda$-modules by the class of quasi-isomorphisms).
In particular, for any objects $\mathcal{F}$ and $\mathcal{G}$ of
$D(Sh_{\h}(X,\Lambda))$, morphisms from $\mathcal{F}$ to $\mathcal{G}$ form
an object $\Hom(\mathcal{F},\mathcal{G})$ of $D(\Lambda)$. The appropriate
version of the Yoneda Lemma thus reads:
\[\Hom(\Lambda(Y),\mathcal{F})\cong \mathcal{F}(Y)\]
for any separated $X$-scheme of finite type $Y$.
In particular, $H^i(Y,\mathcal{F})=H^i(\mathcal{F}(Y))$ is what the old fashioned
literature would call the $i$-th hypercohomology group of $Y$ with coefficients
in $\mathcal{F}$.
\end{paragr}
\begin{paragr}
A sheaf $\mathcal{F}$
is called \emph{$\mathbf{A}^{1}$-local}\index{A1-local@$\mathbf{A}^{1}$-local (sheaf)}
if $\mathcal{F}(Y)\to\mathcal{F}(Y\times\mathbf{A}^{1})$
is an equivalence for all $Y$.
A map $f:M\to N$ is an
\emph{$\mathbf{A}^{1}$-equivalence}\index{A1-equivalence@$\mathbf{A}^{1}$-equivalence}
if for every $\mathbf{A}^{1}$-local $\mathcal{F}$ the map 
\[
f^{*}:\Hom(N,\mathcal{F})\to \Hom(M,\mathcal{F})
\]

is an equivalence. 

Define 
\[
\uDM_h^{\eff}(X,\Lambda)
\]
 to be the localization of $D(Sh_{\h}(X,\Lambda))$ with respect to $\mathbf{A}^{1}$-equivalences.
 We have a localization functor $D(Sh_{h}(X,\Lambda))\to \uDM_{h}^{\eff}(X,\Lambda)$
with fully faithfull right adjoint whose essential image consists of the $\mathbf{A}^{1}$-local
objects. An explicit description of the right adjoint is by taking the total complex
of the bicomplex
\[
C_{*}(\mathcal{F})(Y)=\cdots\to\mathcal{F}(Y\times\Delta_{\mathbf{A}^{1}}^{n})\to\cdots\to\mathcal{F}(Y\times\Delta_{\mathbf{A}^{1}}^{1})\to\mathcal{F}(Y)\, ,
\]
where $\Delta_{\mathbf{A}^{1}}^{n}=Spec(k[x_{0},\ldots,x_{n}]/(x_{0}+\cdots+x_{n}=1))$. 
The $\infty$-category $\uDM_h^{\eff}(X,\Lambda)$ comes equipped with a canonical functor
\[\gamma_X:Sch/X\times D(\Lambda)\to \uDM_h^{\eff}(X,\Lambda)\]
defined by $\gamma_X(Y,K)=\Lambda(Y)\otimes_\Lambda K$.
Furthermore, it is a presentable $\infty$-category (as a left Bousfield localization
of a presentable $\infty$-category, namely $D(Sh_{\h}(X,\Lambda))$), and thus
has small colimits and small limits.
For a cocomplete $\infty$-category $C$, the category of colimit preserving functors
$\uDM_h^{\eff}(X,\Lambda)\to C$ is equivalent to the category of functors
$F:Sch/X\times D(\Lambda)\to C$ with the following two properties:
\begin{itemize}
\item For each $X$-scheme $Y$, the functor $F(Y,-):D(\Lambda)\to C$ commutes with
colimits.
\item For each complex of $\Lambda$-modules $K$, we have:
\begin{itemize}
\item[a)] the first projection induces an equivalence $F(Y\times\mathbf{A}^1,K)\cong F(Y,K)$
for any $X$-scheme $Y$;
\item[b)] for any $h$-hypercovering $U$ of $Y$, the induced map
$\colim_{\Delta^{\mathit{op}}}F(U,K)\to F(Y,K)$ is invertible.
\end{itemize}
\end{itemize}
The functor $\uDM_h^{\eff}(X,\Lambda)\to C$ associated to such an $F$
is constructed as the left Kan extension of $F$ along $\gamma_X$.

There is still an issue. Indeed, let $\infty\in\mathbf{P}^{1}$ and let us form the following cofiber sequence:
\[
\Lambda(X)\stackrel{\infty}{\to}\Lambda(\mathbf{P}^{1})\stackrel{}{\to}\Lambda(1)[2]
\]

In order to express Poincar\'e duality (or, more generally, Verdier duality),
we need the cofiber $\Lambda(1)[2]$ above to be
$\otimes$-invertible. But it is not so. 
\end{paragr}
\begin{defn}
An object $A\in C$ is
\emph{$\otimes$-invertible}\index{tensor-invertible@$\otimes$-invertible}
if the functor $A\otimes-:C\to C$
is an equivalence of $\infty$-categories. 
\end{defn}

We want to invert a non-invertible object. Let us think about the case
of a ring. 

\[
R[f^{-1}]=\colim(R\stackrel{f}{\to}\stackrel{}{R\stackrel{f}{\to}\cdots})
\]
(The colimit is taken within $R$-modules.)
For $\infty$-categories, we define $C[A^{-1}]$
with a similar colimit formula. Note however that
the colimit needs to be taken in the category of presentable $\infty$-categories
(in which the maps are the colimit preserving functors).
We get an explicit description of this colimit as follows.
For $C$ presentable, $C[A^{-1}]$
can be described as the limit of the diagram
\[
\cdots\xrightarrow{\uHom(A,-)}C\xrightarrow{\uHom(A,-)}C\xrightarrow{\uHom(A,-)}C
\]
in the $\infty$-category of $\infty$-categories (here, $\uHom(A,-)$ is the
right adjoint of the functor $A\otimes -$).
Therefore, an object in $C[A^{-1}]$ is typically a sequence $(M_{n},\sigma_{n})_{n\geq 0}$
with $M_n$ objects of $C$ and
$\sigma_{n}:M_{n}\stackrel{\sim}{\to}\uHom(A,M_{n+1})$
equivalences in $C$. Note that, in the case where $A$ is the circle in the $\infty$-category of pointed
homotopy types, we get exactly the definition of an
$\Omega$-spectrum from topology. There is a canonical functor
\[\Sigma^\infty:C\to C[A^{-1}]\]
which is left adjoint to the functor
\[\Omega^\infty:C[A^{-1}]\to C\]
defined as $\Omega^\infty(M)=M_0$ where $M=(M_{n},\sigma_{n})_{n\geq 0}$ is a sequence
as above.

There is still the issue of
having a natural symmetric monoidal structure on $C[A^{-1}]$, which is not automatic.
However, if the cyclic permutation
acts as the identity on $A^{\otimes 3}$ (by permuting the factors) in the homotopy category
of $C$, then there is a unique symmetric monoidal structure on $C[A^{-1}]$ such that
the canonical functor $\Sigma^\infty:C\to C[A^{-1}]$ is
symmetric monoidal (all these issues are very well explained
in Robalo's \cite{robalo}).
Fortunately for us, Voevodsky proved that this extra property
holds for $C=\uDM_h^{\eff}(X,\Lambda)$ and $A=\Lambda(1)$.

\begin{defn}
The big category of $h$-motives\index{h-motive@$h$-motive}
is defined as:
\[
\uDM_{h}(X,\Lambda)=\uDM_{h}^{\eff}(X,\Lambda)[\Lambda(1)^{-1}].
\]
\end{defn}

\begin{rem}
However, what is important here is the universal property
of the stable $\infty$-category $\uDM_{h}(X,\Lambda)$; given a cocomplete $\infty$-category $C$,
together with an equivalence of categories $T:C\to C$
each colimit preserving functor
$\varphi:\uDM_{h}^\eff(X,\Lambda)\to C$ equipped with an invertible
natural transformation $\varphi(M\otimes\Lambda(1)[2])\cong T(\varphi(M))$
is the composition of a unique colimit preserving functor
$\Phi:\uDM_{h}(X,\Lambda)\to C$ equipped with an invertible
natural transformation $\Phi(M\otimes\Sigma^\infty\Lambda(1)[2])\cong T(\Phi(M))$.
\end{rem}

\begin{rem}
{We are very far from having locally constant sheaves here!} In classical settings, the Tate object
$\Lambda(1)$ is locally constant (more generally, for a smooth and proper map $f:X\to Y$
we expect each cohomology sheaf $R^if_*(\Lambda)$ to be locally constant). However the special case
of the projective line shows that we cannot have such a property motivically:
Over field $k$, the cohomology with coefficients in $\mathbf{Q}$ vanishes
in degree $>0$, while, with coefficients in $\mathbf{Q}(1)$,
it is equal to $k^\times\otimes\mathbf{Q}$ in degree $1$.
Therefore we should ask what is the replacement of locally constant sheaves. This will be dealt with later, when we will explain what are
constructible motives.
\end{rem}

\begin{defn}
We have an adjunction 
\[
\Sigma^{\infty}:\uDM_{h}^{\eff}(X,\Lambda)\rightleftarrows \uDM_{h}(X,\Lambda):\Omega^{\infty}
\]
and we define $M(Y)=\Sigma^{\infty}\Lambda(Y)$.
This is the \emph{motive} of $Y$ over $X$, with coefficents in $\Lambda$. 
\end{defn}
As we want eventually to do intersection theory, we need Chern classes within motives.
Here is how they appear.
Consider the morphisms of $h$-sheaves of groups
$\mathbf{Z}(\mathbf{A}^{1}-\{0\})\to\mathbf{G}_{m}$
on the category $Sch/X$
corresponding to the identity $\mathbf{A}^1-\{0\}=\mathbf{G}_m$, seen as a map of sheaves
of sets. From the pushout diagram 
\[
\begin{tikzcd}
\mathbf{A}^{1}-\{0\}\ar{r}\ar{d} & \mathbf{A}^{1}\ar{d}\\
\mathbf{A}^{1}\ar{r} & \mathbf{P}^{1}
\end{tikzcd}
\]

and from the identification $\ZZ\cong\ZZ(\mathbf{A}^1)$,
we get a (split) cofiber sequence 
\[
\mathbf{Z}\to\mathbf{Z}(\mathbf{A}^{1}-\{0\})\to\mathbf{Z}(1)[1]
\]
Since the map $\mathbf{Z}(\mathbf{A}^{1}-\{0\})\to\mathbf{G}_{m}$
takes $\ZZ$ to $0$, it induces a canonical map
$\ZZ(1)[1]\to\mathbf{G}_m$.
\begin{thm}[Voevodsky]
The map $\mathbf{Z}(1)[1]\to\mathbf{G}_{m}$
is an equivalence in the effective category $\uDM_{h}^{\eff}(X,\ZZ)$.\label{chern}
\end{thm}
As a result, we get canonical maps:
\begin{itemize}
\item $h$-hyper\-sheafi\-fi\-ca\-tion:
\[
Pic(X)=H_{Zar}^{1}(X,\mathbf{G}_{m})\to
H^0\Hom_{D(Sh_{h}(X,\Lambda))}(\mathbf{Z},\mathbf{G}_{m}[1])\, ;
\]
\item $\mathbf{A}^1$-lo\-cal\-iza\-tion:
\[
H^0\Hom_{D(Sh_{h}(X,\Lambda))}(\mathbf{Z},\mathbf{G}_{m}[1])
\to H^0\Hom_{\uDM_{h}^{\eff}}(\mathbf{Z},\mathbf{G}_{m}[1])\, ;
\]
\item $\mathbf{P}^1$-stabil\-iza\-tion:
\[
H^0\Hom_{\uDM_{h}^{\eff}(X,\ZZ)}(\mathbf{Z},\mathbf{G}_{m}[1])\to
H^0\Hom_{\uDM_{h}(X,\ZZ)}(\mathbf{Z},\mathbf{Z}(1)[2])\, .
\]
\end{itemize}
By composition this gives us the first motivic Chern classes of line bundles.
\[
c_1: Pic(X)\to H_{M}^{2}(X,\mathbf{Z}(1))
=H^0\Hom_{\uDM_{h}(X,\ZZ)}(\mathbf{Z},\mathbf{Z}(1)[2])
\]
\subsection{Functoriality}
\begin{paragr}
Recall that we have an assignment 
\[
X\mapsto\uDM_{h}(X,\Lambda).
\]
There is a unique symmetric monoidal structure on $\uDM_{h}(X,\Lambda)$
such that the functor $M:Sch_{/X}\to\uDM_{h}(X,\Lambda)$
is monoidal. It has the following properties  (we write $\Lambda=M(X)\cong\Sigma^\infty(\Lambda)$
and $\Lambda(1)=\Sigma^\infty(\Lambda(1))$):
\begin{itemize}
\item $A(1)\cong A\otimes\Lambda(1)$; all functors of interest always commute with the functor $A\mapsto A(1)$.
\item $M(Y\times\mathbf{P}^{1})\cong M(Y)[2]\oplus M(Y)$. 
\item $A(n)= A\otimes\Lambda(n)$ is well defined for all $n\in\ZZ$ (with $\Lambda(n)$ the dual of $\Lambda(-n)$
for $n<0$ and $\Lambda(0)=\Lambda$, $\Lambda(n+1)\cong\Lambda(n)(1)$ for $n\geq 0$).
\item There is an internal $\Hom$ functor $\uHom$.
\end{itemize}
For a morphism $f:X\to Y$ we have $f^{*}:\uDM_{h}(Y,\Lambda)\to\uDM_{h}(X,\Lambda)$
which preserves colimits and thus has right adjoint $f_{*}:\uDM_{h}(X,\Lambda)\to\uDM_{h}(Y,\Lambda)$.
No property of $f$ is required for that. We construct first the functor
\[f^*:\uDM_{h}^{\eff}(Y,\Lambda)\to\uDM^{\eff}_{h}(X,\Lambda)\]
as the unique colimit preserving functor which fits in the commutative diagram
\[\begin{tikzcd}
Sch/Y\times D(\Lambda)\ar{r}{f^*\times 1_{D(\Lambda)}}\ar{d}&Sch/X\times D(\Lambda)\ar{d}\\
\uDM_{h}^{\eff}(Y,\Lambda)\ar{r}{f^*}&\uDM^{\eff}_{h}(X,\Lambda)
\end{tikzcd}\]
(in which the vertical functors are the canonical ones
$(U,C)\mapsto M(U)\otimes_\Lambda C$),
and observe that it has a natural structure of symmetric monoidal functor.
There is thus a unique symmetric monoidal pull-back functor $f^*$ defined on
$\uDM_{h}$ so that the following squares commutes.
\[\begin{tikzcd}
\uDM_{h}^{\eff}(Y,\Lambda)\ar{r}{f^*}\ar{d}{\Sigma^\infty}&\uDM^{\eff}_{h}(X,\Lambda)\ar{d}{\Sigma^\infty}\\
\uDM_{h}(Y,\Lambda)\ar{r}{f^*}&\uDM_{h}(X,\Lambda)
\end{tikzcd}\]
If moreover $f$ is separated
and of finite type then the pull-back functor $f^{*}$ has a left adjoint
functor $f_{\sharp}:\uDM_{h}(X,\Lambda)\to\uDM_{h}(Y,\Lambda)$
which preserves colimits, and is essentially determined by the property that $f_{\sharp}M(U)=M(U)$
for any separated $X$-scheme of finite type $U$ via universal properties as above.
For example $f_{\sharp}(\Lambda)=M(X)$. We have a projection formula (proved by observing that
the formula holds in the category of schemes and then extending by colimits)
\[
f_{\sharp}(A\otimes f^{*}(B))\stackrel{\simeq}{\to}f_{\sharp}A\otimes B.
\]
\end{paragr}
\begin{xca}
Show that, for any Cartesian square of noetherian schemes
\[\begin{tikzcd}
X'\ar{r}{u}\ar{d}{f'}&X\ar{d}{f}\\
Y'\ar{r}{v}&Y
\end{tikzcd}\]
and for any $M$ in $\uDM_{h}(X,\Lambda)$, if $v$ is
separated of finite type, then the canonical map
\[v^*f_*(M)\to f'_*u^*(M)\]
is invertible.
\end{xca}
The base change formula above is too much: we want this to hold only
for $f$ proper of $v$ smooth, because, otherwise, we will not have any
good notion of support of a motive. This is why we have to restrict
ourselves to a subcatefgory of $\uDM_{h}(X,\Lambda)$, on which
the support will be well defined.
\begin{defn}
Let $\DM_{h}(X,\Lambda)$ be the smallest full subcategory of $\uDM_{h}(X,\Lambda)$
closed under small colimits, containing objects of the form $M(U)(n)[i]$
for $U\to X$ smooth and $i,n\in\ZZ$. 
\end{defn}
\begin{rem}
The $\infty$-category $\DM_{h}(X,\Lambda)$ is stable  and presentable, essentially by construction.
It is also stable under the operator $M\mapsto M(n)$ for all $n\in \ZZ$.
\end{rem}
\begin{thm}[Localization Property]\index{localization property}
Take $i:Z{\to X}$ to be a closed emdebbing
with open complement $j:U{\to X}$ and let $M\in \DM_{h}(X,\Lambda)$.
Then we have a canonical cofiber sequence (in which the maps
are the co-unit and unit of appropriate adjunctions):
\[
j_{\sharp}j^{*}M\to M\to i_{*}i^{*}M
\]
\end{thm}
Idea of the proof: the functors $j_\sharp$, $j^*$, $i_*$ and $i^*$
commute with colimits. Therefore, it is sufficient to prove the case where $M=M(U)$
with $U/X$ smooth. We conclude by an argument due to Morel and Voevodsky,
using Nisnevich excision as well as the fact, locally for the Zariski
topology, $U$ is \'etale on $\mathbf{A}^n\times X$. Then, using Nisnevich
excision, we reduce to the vase where  $U=\mathbf{A}^n\times X$, in which
case we can provide explicit $\mathbf{A}^1$-homotopies.

\begin{xca}
Show that $j_{\sharp}j^{*}M\to M\to i_{*}i^{*}M$
is not a cofiber sequence in $\uDM_{h}(X,\Lambda)$ for an arbitrary object $M$.
\end{xca}
The functor $f^{*}$ restricts to a functor on $\DM_h$,
and also for $f_{\sharp}$ if $f$ is smooth. Moreover, $\DM_{h}$ is closed
under tensor product. If $i:Z\to X$ is a closed immersion, than by the cofiber
sequence above we see that the functor $i_{*}$ sends $\DM_h(Z,\Lambda)$ to
$\DM_h(X,\Lambda)$. 

\begin{rem}
By presentability, the inclusion
$\DM_{h}(X,\Lambda)\stackrel{i}{\to}\uDM_{h}(X,\Lambda)$
has right adjoint $\rho$. 

For $f:X\to Y$ we define 
\[
f_{*}:DM_{h}(X,\Lambda)\to DM_{h}(Y,\Lambda)
\]
 by 
\[
f_{*}M=\rho f_{*}i(M)
\]
We can use this to describe the internal Hom as well:
\[\uHom(A,B)=\sigma \uHom(i(A),i(B))\, .\]
\end{rem}
\begin{prop}
For any embedding $i:Z\to X$ the functors $i_{*},i_{\sharp}$ are
both fully faithful. 
\end{prop}
Using this and some abstract nonsense we get that $i_{*}$ has a right
adjoint $i^{!}$ and there are canonical fiber sequences 
\[
i_{*}i^{!}M\to M\to j_{*}j^{*}M
\]
We also have a smooth base change formula\index{base change formula!smooth}
and
a proper base change formula\index{base change formula!proper}:
\begin{thm}[Ayoub, Cisinski-D\'eglise]
For any Cartesian square of noetherian schemes
\[\begin{tikzcd}
X'\ar{r}{u}\ar{d}{f'}&X\ar{d}{f}\\
Y'\ar{r}{v}&Y
\end{tikzcd}\]
and for any $M$ in $\DM_{h}(X,\Lambda)$, if either $v$ is
separated smooth of finite type, or if $f$ is proper, then the canonical map
\[v^*f_*(M)\to f'_*u^*(M)\]
is invertible in $\DM_h(X,\Lambda)$.
\end{thm}
The proof follows from Ayoub's axiomatic approach \cite{ay1},
under the additional assumption that all the maps are quasi-projective.
The general case may be found in \cite[Theorem 2.4.12]{CD3}.
\begin{defn}[Deligne]
Let $f:X\to Y$ be separated of finite type, or equivalently, by Nagata's
theorem, assume that there is a relative compactification which is
a factorization of $f$ as

\[
X\stackrel{j}{\to}\bar{X}\stackrel{p}{\to}Y\, ,
\]
 where $j$ is an open embedding and $p$ is proper. Then we define 
\[
f_{!}=p_{*}j_{\sharp}
\]
\end{defn}
Here are the main properties we will use (see \cite{CD3}):
\begin{itemize}
\item The functor $f_{!}$ admits a right adjoint $f^{!}$ (because it commutes with colimits). 
\item There is a comparison map $f_{!}\to f_{*}$ constructed as follows. There is
a map $j_\sharp\to j_*$ which corresponds by transposition to the inverse of the
isomorphism from $j^*j_*$ to the identity due to the fully faithfulness of $j_*$.
Therefore we have a map $f_!=p_* j_\sharp\to p_*j_*\cong f_*$.
\item Using the proper base change formula, we can prove that push-forwards with compact support\index{push-forward!with compact support}
are well defined: in particular, the functor $f_!$ does not depend on the choice of the
compactification of $f$ up to isomorphism. Furthermore, if $f$ and $g$
are composable, there is a coherent isomorphism $f_!g_!\cong (fg)_!$.
\end{itemize}
The proof of the proper base change formula relies
heavily on the following property.
\begin{thm}[Relative Purity]\index{purity!relative}\index{relative purity}
If $f:X\to Y$ is smooth and separated of finite type, then 
\[
f^{!}(M)\cong f^{*}(M)(d)[2d]
\]

where $d=dim(X/Y)$.
\end{thm}
The first appearance of this kind of result in a motivic context
(i.e. in stable homotopy category of schemes) was in a
preprint of Oliver R\"ondigs~\cite{roendigs}.
As a matter of facts, the proof of relative purity
can be made with a great level of generality, as
in Ayoub's thesis \cite{ay1}, where we see that the only
inputs are the localization theorem and $\mathbf{A}^1$-homotopy invariance.
However, in our situation
(where Chern classes are available),
the proof can be dramatically simplified
(see the proof \cite[Theorem 4.2.6]{CD4},
which can easily be adapted to the context of $h$-sheaves).
A very neat and robust proof (in equivariant stable homotopy
category of schemes, but which may be seen in any context with the
six operations) may be found in Hoyois' paper \cite{hoyois}.
\begin{rem}
For a vector bundle $E\to X$ of rank $r$, we can define its \emph{Thom space} $Th(E)$
by the cofiber sequence
\[\Lambda(E-0)\to\Lambda(E)\to Th(E)\]
(where $E-0$ is the complement of the zero section).
Using motivic Chern classes, we can construct the Thom isomorphism
\[
Th(E)\cong\Lambda(r)[2r]\, .
\]

What is really canonical and conceptually right is 
\[
f^{!}(M)\cong f^{*}(M\otimes Th(T_{f})).
\]
We refer to Ayoub's work for more details.
From this we can deduce a formula relating $f_{!}$ and $f_{\sharp}$
when $f$ is smooth. By transposition, relative putity takes the following form.
\end{rem}
\begin{cor}
If $f:X\to Y$ is smooth and separated of finite type then 
\[
f_{\sharp}(M)\cong f_{!}(M)(d)[2d].
\]
\end{cor}
Finally, we also need the Projection Formula (see \cite[Theorem 2.2.14]{CD3}): 
\begin{prop}\index{projection formula}
If $f:X\to Y$ is separated of finite type then,
for any $A$ in $\DM_h(X,\Lambda)$ and any $B$ in
$\DM_h(Y,\Lambda)$, there is a canonical
isomorphism:
\[
f_{!}(A\otimes f^{*}B)\cong f_{!}(A)\otimes B\, .
\]
\end{prop}

\begin{xca}
\end{xca}
\begin{itemize}
\item Let $f:X\to Y$, then $f_{*}\uHom(f^{*}M,N)\cong \uHom(M,f_{*}N)$. 
\item For $f$ separated of finite type we have $\uHom(f_{!}M,N)\cong f_{*}\uHom(M,f^{!}N)$. 
\item For $f$ as above, $f^{!}\uHom(M,N)\cong \uHom(f^{*}M,f^{!}N)$
\item For $f$ \emph{smooth}, $f^{*}\uHom(M,N)\cong \uHom(f^{*}M,f^{*}N)$. 
\end{itemize}
A reformulation of the proper base change formula is the following.
\begin{thm}\index{base change formula!proper}
For any pull-back square of noetherian schemes
\[
\begin{tikzcd}
X'\ar{r}{u}\ar{d}{f'} & X\ar{d}{f}\\
Y'\ar{r}{v} & Y
\end{tikzcd}
\]
with $f$ is separated of finite type, we have $v^{*}f_{!}\cong f'_{!}u^{*}$
and $f^{!}v_{*}\cong u_{*}(f')^{!}$. 
\end{thm}
\begin{rem}
Given a morphism of rings of coefficients $\Lambda\to\Lambda'$, there
is an obvious change of coefficients functor
\[\DM_h(X,\Lambda)\to\DM_h(X,\Lambda')\ , \quad
M\mapsto\Lambda'\otimes_\Lambda M\]
which is symmetric monoidal and commutes with the four
operations $f^*$, $f_*$, $f^!$ and $f_!$ whenever they are defined.
Moreover, one can show that an object $M$ in $\DM_h(X,\ZZ)$
is null if and only if $\QQ\otimes M\cong 0$
and $\ZZ/p\ZZ\otimes M\cong 0$ for any prime number $p$;
see \cite[Prop.~5.4.12]{CD4}.
Fortunately, $\DM_h(X,\Lambda)$ may be understood in more
tractable terms whenever $\Lambda=\QQ$ of $\Lambda$ is finite,
as we will see in the next section.
\end{rem}
\subsection{Representability theorems}
\begin{paragr}
We define \emph{\'etale motivic cohomology}\footnote{Also known
as Lichtenbaum cohomology.}\index{motivic cohomology!\'etale}
of $X$ with coefficients in $\Lambda$ as
\[H^i_M(X,\Lambda(n))=H^i(\Hom_{DM_{h}(X,\Lambda)}(\Lambda,\Lambda(n)))\]
for all $i,n\in\ZZ$.
\end{paragr}
\begin{thm}[Suslin-Voevodsky, Cisinski-D\'eglise]\index{K-theory@$K$-theory}
For any noetherian scheme of finite dimension $X$,
\[
H_{M}^{i}(X,\mathbf{Q}(n))\cong(KH_{2n-i}(X)\otimes\mathbf{Q})^{(n)}
\]
where $KH$ is the homotopy invariant K-theory of Weibel and the $"^{(n)}"$
stands for the fact that we take the intersection of the $k^{n}$-eigen-spaces
of the Adams operations $\psi_{k}$ for all $k$. For $X$ regular, we simply have
$H_{M}^{i}(X,\mathbf{Q}(n))\cong(K_{2n-i}(X)\otimes\mathbf{Q})^{(n)}$. In particular,
for $X$ regular and $n\in\ZZ$, we have:
\[ CH^n(X)\otimes\QQ\cong H_{M}^{2n}(X,\mathbf{Q}(n))\, .\]
\end{thm}
The case where $X$ is separated and smooth of finite type over a field
is due to Suslin and Voevodsky (puting together the results
of \cite{SV} and of \cite{FSV}).
The general case follows from {\cite[Theorem~5.2.2]{CD4}},
using the representability theorem of $KH$
announced in \cite{voe} and proved in \cite{Cis4}. More generally,
one may recover motivically
$\mathbf Q$-linear Chow groups of possibly singular schemes
as well as Bloch's higher Chow groups as follows.
\begin{thm}[motivic cycle class]\index{motivic cohomology!and Chow groups}
\index{motivic cohomology!and higher Chow groups}\index{cycle class!motivic}
Let $f:X\to Spec(k)$ be separated of finite type. Then 
\[
H^0(\Hom_{DM_{h}(X,\mathbf{Q})}(\mathbf{Q}(n)[2n],f^{!}\mathbf{Q}))\cong CH_{n}(X)\otimes\mathbf{Q}
\]

and, if $X$ is equidimensional of dimension $d$, then\label{cycle class}
\[
H^0(\Hom_{\DM_{h}(X,\mathbf{Q})}(\mathbf{Q}(n)[i],f^{!}\mathbf{Q}))\cong CH^{d-n}(X,i-2n)\otimes\mathbf{Q}\, .
\]
\end{thm}
This follows from
\cite[Corollaries 8.12 and 8.13, Remark 9.7]{CD5}.
These representability result may be used to see how classical
Grothendieck motives of smooth projective varieties
over a field $k$
may be seen in this picture: they form the full subcategory
of $\DM_{h}(\spec k,\mathbf{Q})$ whose objects are the
direct factors of motives of the form $M(U)(n)$ with $U$
smooth and projective over $k$ and $n\in\ZZ$).
The following statement is known as
\emph{rigidity theorem}\index{rigidity theorem}
\begin{thm}[Suslin-Voevodsky, Cisinski-D\'eglise]
Given a locally noetherian scheme $X$, there is a canonical equivalence
of $\infty$-categories
\[
DM_{h}(X,\Lambda)\cong D(Sh(X_{et},\Lambda))
\]
for $\Lambda$ of positive invertible characteristic on $X$, compatible
with 6-operations. In particular\label{rigidity theorem}
\[
H_{M}^{i}(X,\Lambda(j))\cong H_{et}^{i}(X,\mu_{n}^{\otimes j}\otimes\Lambda).
\]
\end{thm}
The case where $X$ is the spectrum of a field is essentially contained in
the work of Suslin and Voevodsky \cite{SV}.
See \cite[Corollary~5.5.4]{CD4} for the general case.
We should mention that the equivalence of categories above
is easy to construct. The main observation is Voevodsky's theorem \ref{chern},
together with the Kummer short exact sequence induced by $t\mapsto t^n$
\[0\to\mu_n\to\mathbf{G}_m\to\mathbf{G}_m\to 0\]
(where $\mu_n$ is the sheaf of $n$-th roots of unity),
from which follows the identification
$\Lambda(1)\cong\mu_n\otimes_{\ZZ/n\ZZ}\Lambda$, where $n$ is the characteristic of $\Lambda$.
In particular, $\Lambda(1)$ is already $\otimes$-invertible, which implies
(by inspection of universal properties) that
\[\uDM^{\eff}_h(X,\Lambda)\cong\uDM_h(X,\Lambda)\, .\]
On the other hand, $\uDM^{\eff}_h(X,\Lambda)$ is a full subcategory
of the derived category of $h$-sheaves of $\Lambda$-modules.
The comparison functor from $\uDM^{\eff}_h(X,\Lambda)$ to $D(Sh(X_{et},\Lambda))$
is simply the restriction functor. The precise formulation of the
previous theorem is that the composition
\[\DM_h(X,\Lambda)\subset\uDM_h(X,\Lambda)\cong\uDM_h^{\eff}(X,\Lambda)\to D(Sh(X_{et},\Lambda))\]
is an equivalence of $\infty$-categories.
\begin{rem}
If $char(\Lambda)=p^{i}$ then one proves that
$DM_{h}(X,\Lambda)\cong DM_{h}(X[\frac{1}{p}],\Lambda)$
(using the Artin-Schreier short exact sequence together with the localization
property) so that we can assume that the ring of functions on $X$ always has the
characteristic of $\Lambda$ invertible in it; see \cite{CD4}. \label{rigidity artin}
\end{rem}
\begin{rem}
One can have access to $H_{M}^{i}(X,\ZZ(n))$ via the coniveau spectral sequence
whose $E_1$ term is computed as Cousin complex, and thus gives rise to a nice and
rather explicit theory of residues; see \cite[(7.1.6.a) and Prop. 7.1.10]{CD4}.
\end{rem}
\section{Finiteness and Euler characteristic}

\subsection{Locally constructible motives}
\begin{paragr}
Recall that an object $X$ in a tensor category $C$ is
\emph{dualizable}\index{dualizable}
(we also say \emph{rigid}) if there exists $Y\in C$ such that $X\otimes-$ is left adjoint to
$Y\otimes-$. This provides an isomorphism $Y\cong \uHom(X,1_{C})$. In other words
$Y\otimes a\cong \uHom(X,a)$. This way, we get the evaluation map $\epsilon:Y\otimes X\to1_{C}$
and as well as the co-evaluation map $\eta:1_{C}\to X\otimes Y$. 
This exhibits the adjunction between the tensors. In particular, composing
$\epsilon$ and $\eta$ approriately tensored by $X$ or $Y$ gives the identity:
\[1_X: X\to X\otimes Y\otimes X\to X\quad\text{and}\quad
1_Y: Y\to Y\otimes X\otimes Y\to Y\, .\]
\end{paragr}
\begin{rem}\label{rem:dualizable functorial}
If $F:C\to D$ is a monoidal functor, if $x\in C$ dualizable then so is $F(x)$,
and $F(x^{\wedge})\cong F(x)^{\wedge}$.
Furthermore, $F$ also preserve internal $\uHom$ from $x$, since $\uHom(x,y)\cong x^\wedge\otimes y$
for all $y$.
\end{rem}
\begin{rem}
If $C\in D(Sh_{et}(X,\Lambda))$ then it is dualizable
if and only if it is locally constant with perfect fibers;
see \cite[Remark 6.3.27]{CD4}.
That means that $C$ is dualizable if and only if the following condition
holds: there is a surjective \'etale map $u:X'\to X$ together with a
perfect complex of $\Lambda$-modules $K \in Perf(\Lambda)$
(i.e. complex of $\Lambda$-modules $K$
which is quasi-isomorphic to a bounded
complex of projective $\Lambda$-modules
of finite type), and an isomorphism $K_ {X'}\cong u^*(C)$
in $D(Sh_{et}(X',\Lambda))$, where $K_ {X'}$ is the constant sheaf
on $X'$ associated to $K$.
\end{rem}
\begin{paragr}
Suppose $1/n\in\mathcal{O}_{X}$, $n=char(\Lambda)>0$.
Then $DM_{h}(X,\Lambda)\cong D(Sh_{et}(X,\Lambda))$.
Inside it, we have the subcategory $D_{ctf}^{b}(X_{et},\Lambda)$ of
constructible sheaves finite
tor-dimension\index{constructible!sheaves with finite tor dimension}.
If there is $d$ such that $cd(k(x))\le d$
for every point $x$ of $X$, then it is simply the subcategory of compact objects.
In general, this subcategory $D_{ctf}^{b}(X_{et},\Lambda)$ is important because
it is closed under the six operations. We look for correspondent in motives
with arbitrary ring of coefficients $\Lambda$.
We can characterise those \'etale sheaves by 
\[
\{C\in D(Sh_{et}(X,\Lambda))\, |\,
\exists\text{ stratification }X_{i}:C_{|X_{i}}\text{ locally constant with perfect fibers}\}
\]
Namely, an object $C$ of $D(Sh_{et}(X,\Lambda))$ is constructible of finite tor-dimension
if and only if there exists a finite stratification of $X$ by locally closed
subschemes $X_i$ together with $\phi_{i}:U_{i}\to X_{i}$ \'etale surjective for each $i$, and
there is $K_{i}\in Perf(\Lambda)$ (compact objects in the
derived category of $\Lambda$-modules),
and an isomorphism
$\phi_{i}^{*}(C_{|X_{i}})\cong({K}_{i})_{X_i}$ in the derived category of
sheaves of $\Lambda$-modules on the small \'etale site of $X$;
see~\cite[Remark~6.3.27]{CD4}. \label{par:ctf}
\end{paragr}
\begin{xca}[Poincar\'e Duality]\index{Poincar\'e duality}\index{duality!Poincar\'e}
Let $f:X\to Y$ be smooth and proper of relative dimension $d$.
Then, if $M\in DM_{h}(X,\Lambda)$ is dualizable, so is $f_{*}(M)$ and
\[
f_{*}(M)^{\wedge}\cong f_{*}(M^{\wedge})(-d)[-2d]
\]
with $M^\wedge=\uHom(M,\Lambda)$ the dual of $M$.
\end{xca}
\begin{defn} The $\infty$-category
$DM_{h,c}(X,\Lambda)$ of \emph{constructible}\index{motive!constructible}
\index{constructible!motive}
$\Lambda$-linear \'etale motives over $X$ is the smallest thick subcategory (closed under
shifts, finite colimits and retracts) containing $f_{\sharp}(\Lambda)(n)$
for any $f:U\to X$ smooth and every $n\in\mathbf{Z}$. 
\end{defn}
The following proposition is
an easy consequence of relative purity and of the proper base change formula.

\begin{prop} The $\infty$-category
$\DM_{h,c}(X,\Lambda)$ is equal to each of the following subcategories
of $\DM_{h}(X,\Lambda)$:
\begin{itemize}
\item The smallest thick subcategory containing $f_{*}(\Lambda)(n)$ for
$f:U\to X$ proper and $n\in\ZZ$. 
\item The smallest thick subcategory containing $f_{!}(\Lambda)(n)$ for
$f:U\to X$ separated of finite type and $n\in\ZZ$. 
\end{itemize}\label{proper devissage}
\end{prop}

\begin{thm}[Absolute Purity]\index{purity!absolute}\index{absolute purity}
If $i:Z\to X$ is a closed emmersion
and assume that both $X,Z$ are regular. Let $c=codim(Z,X)$. Then
there is a canonical isomorphism 
\[
i^{!}(\Lambda_{X})\cong\Lambda_{Z}(-c)[-2c].
\]
\end{thm}
See \cite[Theorem 5.6.2]{CD4}
\begin{rem}
Modulo the rigidity theorem \ref{rigidity theorem},
the proof for the case of finite coefficients is due to Gabber and
was known for a while,
with two different proofs \cite{Fuj,ILO} (although, in characteristic zero, this
goes back to Artin in SGA~4).
After formal reductions using deformation to the normal cone,
one sees that, in order to prove the
absolute purity theorem above, it is then sufficient to consider the case where $\Lambda=\QQ$.
The idea is then that Quillen's localization fiber sequence
\[
\begin{tikzcd}
K(Z)\ar{r}\ar{d}{\wr} & K(X)\ar{r}\ar{d}{\wr} & K(X-Z)\ar{d}{\wr}\\
K(Coh(Z))\ar{r} & K(Coh(X))\ar{r} & K(Coh(X-Z))
\end{tikzcd}
\]
induces a long exact sequence which we may tensor with $\QQ$, and
Absolute purity is then proved using the representability theorem of $K$-theory in the motivic
stable homotopy category together with a variation on the Adams-Riemann-Roch
theorem.
\end{rem}
We recall that a locally noetherian scheme $X$ is
\emph{quasi-excellent}\index{quasi-excellent scheme}
if the following two conditions are verified:
\begin{itemize}
\item[1.] For any point $x\in X$, the completion
map $\mathcal{O}_{X,x}\to\hat{\mathcal{O}}_{X,x}$
is regular (i.e., for any field extension $K$ of the residue field $\kappa(x)$,
the noetherian ring $K\otimes_{\kappa(x)}\hat{\mathcal{O}}_{X,x}$ is regular).
\item[2.] For any scheme of finite presentation $Y$ over $X$,
there is a regular dense open subscheme $U\subset Y$.
\end{itemize}
A locally noetherian scheme is \emph{excellent}
if it is quasi-excellent and universally catenary.
In practice,
what needs to be known is that
any scheme of finite type over a quasi-excellent scheme is quasi-excellent, and
$Spec(R)$ is excellent whenever $R$ is either a field or the ring of integers
of a number field (note also that noetherian complete local rings are
excellent).

\begin{thm}[de Jong-Gabber~{\cite{ILO}}] Any
quasi-excellent scheme is regular locally
for the $h$-topology. In other words, for any quasi-excellent scheme $X$, there
exists an $h$-covering $\{X_i\to X\}_i$ with each $X_i$ regular.
Furthermore, locally for the $h$-topology any
nowhere dense closed subscheme of $X$ is either empty of a divisor with normal crossings:
given any nowhere dense closed subscheme $Z\subset X$, we may choose
the covering above such that the pullback of $Z$ in each $X_i$ is either empty or
a divisor with normal crossings.
Even better, given a prime $\ell$ invertible in $\mathcal{O}_X$,
we may always choose $h$-coverings
$\{X_i\to X\}_i$ as above such that, for each point $x\in X$, there exists an $i$ and there
exists $x_{i}\in X_{i}$ such that $p_{i}(x_{i})=x$ and such that
$[k(x_{i}):k(x)]$ is prime to $\ell$.\label{resolution of singularities}
\end{thm}

\begin{rem}
One can show that the category $DM_{h,c}(X,\Lambda)$ is preserved
by the 6 operations. However, there is a drawback: unless we make finite
cohomological dimension assumptions,
the category $DM_{h,c}$ in not always a sheaf for the \'etale topology! Here is
its \'etale sheafification (which can be proved to be a sheaf of $\infty$-categories
for the $h$-topology).
\end{rem}
\begin{defn}
A motivic sheaf $M$ is in $\DM_{h}(X,\Lambda)$
is \emph{locally constructible}\index{locally constructible}\index{motive!locally constructible}
if there
is an \'etale surjection $f:U\to X$ such that $f^{*}M\in DM_{h,c}(X,\Lambda)$. 

Denote the full subcategory of locally constructible
motives by $\DM_{h,lc}(X,\Lambda)$. \label{def:locallyconstr}
\end{defn}
\begin{rem}
If $\QQ\subset \Lambda$, then $\DM_{h,c}(X,\Lambda)=\DM_{h,lc}(X,\Lambda)$
simply is the full subcategory of compact objects in $\DM_{h}(X,\Lambda)$;
see \cite[Prop.~6.3.3]{CD4}.
\end{rem}
\begin{thm}[Cisinski-D\'eglise]
The equivalence $\DM_{h}(X,\Lambda)\cong D(X_{et},\Lambda)$ restricts to an equivalence of
$\infty$-categories
\[
\DM_{h,lc}(X,\Lambda)\cong D_{ctf}^{b}(X,\Lambda)
\]
whenever $\Lambda$ is noetherian of positive characteristic $n$, with $\frac{1}{n}\in\mathcal{O}_X$.\label{rigidity}
\end{thm}
See \cite[Theorem~6.3.11]{CD4}.

For any morphism of noetherian schemes $f:X\to Y$, the functor
$f^*$ sends locally constructible $h$-motives to locally constructible $h$-motives,
and, in the case where $f$ is separated of finite type, so does
the functor $f_!$.
The theorem of de Jong-Gabber above, together with Absolute Purity, are the main ingredients in the proof
of the following finiteness theorem.
\begin{thm}[Cisinski-D\'eglise]
The six operations preserve locally const\-ruct\-ible
$\h$-mo\-tives, at least when restricted
to separated morphisms of finite type between quasi-excellent noetherian schemes of finite
dimension:
\begin{enumerate}
\item for any such scheme $X$ and any locally constructible $\h$-motives $M$ and $N$ over $X$,
the $\h$-motives $M\otimes N$ and $\sHom(M,N)$ are locally constructible;
\item for any morphism of finite type $f:X\to Y$ between quasi-excellent noetherian schemes of finite
dimension, the four functors $f^*$, $f_*$, $f_!$, and $f^!$ preserve the property
of being locally constructible.
\end{enumerate}
\end{thm}
See \cite[Corollary~6.3.15]{CD4}.
\begin{thm}[Cisinski-D\'eglise]\index{locally constructible}\index{motive!locally constructible}
Let $X$ be a noetherian scheme of finite dimension, and $M$ an object
of $\DM_h(X,\Lambda)$.
\begin{enumerate}
\item If $M$ is dualizable, then it is locally constructible.
\item If there exists a closed immersion $i:Z\to X$ with open complement $j:U\to X$
such that $i^*(M)$ and $j^*(M)$ are locally constructible, then $M$ is locally constructible.
\item If $M$ is locally constructible over $X$, then there exists
a dense open immersion $j:U\to X$ such that $j^*(M)$ is dualizable in $\DMhlc(U)$.
\end{enumerate}\label{thm:generic rigidity}
\end{thm}
This is a reformulation of (part of) \cite[Theorem~6.3.26]{CD4}.
\begin{rem}
In particular, an object $M$ of $\DM_{h}(X,\Lambda)$ is constructible if and only if
there exists a finite stratification of $X$ by locally closed subschemes $X_i$
such that each restriction $M_{|X_i}$ is dualizable in $\DM_{h}(X_i,\Lambda)$.
This may be seen as an independence of $\ell$ result. Indeed,
as we will recall below, there are $\ell$-adic realization functors
and they commute with the six functors. In particular, for
each appropriate prime number $\ell$, the $\ell$-adic realization
$R_\ell(M)$ is a constructible $\ell$-adic sheaf:
each restriction $R_\ell(M)_{|X_i}$ is smooth (in the
language of SGA~4, `localement constant tordu')\footnote{It is standard terminology to call such $\ell$-adic sheaves `lisses'. This comes from Deligne's work, which is written in French. I prefer to translate into `smooth' because this is what we do for morphisms of schemes.
The reason is that this terminology comes from the fact that
there are transersality conditions one can define between (motivic or $\ell$-adic) sheaves and morphisms of schemes, and that a basic intuition about smoothness is that a smooth object is transverse to anything: indeed, a smooth sheaf is transverse to any morphism, while any sheaf is transverse to a smooth morphism. This why I think it is better to use the same word to express the smoothness of both sheaves and morphisms of schemes.}, where
the $X_i$ form a stratification of $X$ which is
given independently of $\ell$.
Furthermore, if we apply any of the six operations to $R_\ell(M)$
in the $\ell$-adic context, then we obtain an object of the from
$R_\ell(N)$ for some locally constructible motive $N$,
and thus a stratification as above relatively to $R_\ell(N)$
which does not depend on $\ell$.\label{rem independence stratifications}
\end{rem}

\subsection{Integrality of traces and rationality of $\zeta$-Functions.}

\begin{paragr}
For $x$ a dualizable object in a tensor category $C$ with
unit object $\mathbf 1$, we can from
the trace of an endomorphism. Indeed the trace\index{trace} of $f:x\to x$
is the map $Tr(f):\mathbf{1}\to\mathbf{1}$ defined
as the composite bellow.

\[
\mathbf{1}\xrightarrow{\text{unit}}
\uHom(x,x)\cong x^{\wedge}\otimes x\xrightarrow{1\otimes f}
x^{\wedge}\otimes x\xrightarrow{\text{evaluation}}\mathbf{1}\]

If a functor $\Phi:C\to D$ is symmetric monoidal, then
the induced map
\[\Phi:\Hom_C(x,x)\to\Hom_C(\mathbf{1},\mathbf{1})\]
preserves the formation
of traces: $\Phi(Tr(f))=Tr(\Phi(f))$. 
\end{paragr}
\begin{paragr}
If $M\in \DM_{h,lc}(Spec(k),\Lambda)$ for $k$ a field (see Def.~\ref{def:locallyconstr}), then $M$
is dualizable. Furthermore, the unit is $\Lambda$ and
we can compute
\[H^0\Hom_{\DM_{h,lc}(Spec(k),\Lambda)}(\Lambda,\Lambda)
=\Lambda\otimes\ZZ[1/p]\]
where $p$ is the exponent characteristic of $k$ (i.e.
$p=char(k)$ if $char(k)>0$ or $p=1$ else).
For $f:M\to M$ any map in $\DM_{h,lc}(Spec(k),\ZZ)$,
we thus have its trace
\[Tr(f)\in\ZZ[1/p]\, .\]
The \emph{Euler characteristic}\index{Euler characteristic}
of a dualizable object
$M$ of $\DM_{h}(Spec(k),\ZZ)$ is defined as the trace of its
identity:
\[\chi(M)=Tr(1_M)\, .\]
For separated $k$-scheme of finite type $X$, we define
in particular
\[\chi_c(X)=\chi(a_!\ZZ)\]
with $a:X\to Spec(k)$ the structural map.
\end{paragr}
\begin{paragr}
Let $X$ be a noetherian scheme and $\ell$ a prime number.
Let $\ZZ_{(\ell)}$ be the localization of $\ZZ$ at the prime ideal $(\ell)$.
We may identify $\DM_h(X,\QQ)$ as the full subcategory
of $\DM_h(X,\ZZ_{(\ell)})$ whose objects are the motives $M$ such that
$M/\ell M\cong 0$, where $M/\ell M\cong\ZZ/\ell\ZZ\otimes M$ is defined via the
following cofiber sequence:
\[M\xrightarrow{\ell}M\to M/\ell M\, .\]
We define
\[\hat D(X,\ZZ_\ell)=\DM_h(X,\ZZ_{(\ell)})/\DM_h(X,\QQ)\, .\]
In other words, $\hat D(X,\ZZ_\ell)$ is the localization (in the sense
of $\infty$-categories) of $\DM_h(X,\ZZ_{(\ell)})$ by the maps
$f:M\to N$ whose cofiber is uniquely $\ell$-divisible
(i.e. lies in the subcategory $\DM_h(X,\QQ)$).
One can show that, if $\frac{1}{\ell}\in\mathcal{O}_X$,
the homotopy category of $\hat D(X,\ZZ_\ell)$ is
Ekedahl's derived category of $\ell$-adic sheaves\index{l-adic@$\ell$-adic!sheaves}
on the small \'etale site of $X$.
In fact, as explained in \cite[Prop. 7.2.21]{CD4} (although in the language of
model categories), the rigidity theorem  \ref{rigidity theorem} may be interpreted
as an equivalence of $\infty$-categories of the form:
\[\hat D(X,\ZZ_\ell)\cong\lim_n D(X_{et},\ZZ/\ell^n\ZZ)\]
(here, the limit is taken in the $\infty$-categories of $\infty$-categories).
We thus have a canonical 
\emph{$\ell$-adic realization functor}\index{l-adic@$\ell$-adic!realization}
\[R_\ell:\DM_h(X,\ZZ)\to\lim_n D(X_{et},\ZZ/\ell^n\ZZ)\]
which sends a motive $M$ to $M\otimes\ZZ_{(\ell)}$, seen in the
Verdier quotient $\hat D(X,\ZZ_\ell)$.
We observe that there is a unique way to define the six operations
on $\hat D(X,\ZZ_\ell)$ in such a way that the $\ell$-adic
realization functor commutes with them. In particular,
there is a symmetric monoidal structure on $\hat D(X,\ZZ_\ell)$.

Classically, one defines $D^b_c(X_{et},\ZZ_\ell)$ as the
full subcategory of $\lim_n D(X_{et},\ZZ/\ell^n\ZZ)$
whose objects are the $\ell$-adic systems $(\mathcal{F}_n)$
such that each $\mathcal{F}_n$
belongs to the subcategory
$D^b_{ctf}(X_{et},\ZZ/\ell^n\ZZ)$ (see \ref{par:ctf}).
Furthermore, an $\ell$-adic system $(\mathcal{F}_n)$
is dualizable is and only if $\mathcal{F}_1$ is
dualizable in $D^b_{ctf}(X_{et},\ZZ/\ell\ZZ)$: this
is due to the fact, that, by definition, the canonical functor
\[\hat D(X,\ZZ_\ell)\to D(X_{et},\ZZ/\ell\ZZ)\]
is symmetric monoidal, conservative, and commutes with the formation
of internal Hom's.
In other words, $D^b_c(X_{et},\ZZ_\ell)$ may be identified with the
full subcategory of $\hat D(X,\ZZ_\ell)$ whose objects are those
$\mathcal{F}$ such that there exists a finite stratification by
locally closed subschemes $X_i\subset X$ such that each restriction
$\mathcal{F}_{|X_i}$ is dualizable in $\hat D(X_i,\ZZ_\ell)$.
We thus have a canonical equivalence of $\infty$-categories:
\[D^b_c(X_{et},\ZZ_\ell)
\cong\lim_n D^b_{ctf}(X_{et},\ZZ/\ell^n\ZZ)\, .\]
This implies right away that the six operations restrict
to $D^b_c(X_{et},\ZZ_\ell)$ (if we consider quasi-excellent
schemes only), and, by vitue of Theorem \ref{rigidity}
that we have an $\ell$-adic realization\index{l-adic@$\ell$-adic!realization}
functor
\[R_\ell:\DM_{h,lc}(X,\ZZ)\to D^b_c(X_{et},\ZZ_\ell)\]
which commute with the six operations. For a scheme $X$ with
structural map $a:X\to\spec k$ separated and of finite type,
the motive of $X$ is $M(X)=a_!a^!(\ZZ)$. This is a dualizable object
with dual $a_*(\ZZ)$. Hence
\[
R_\ell(M(X)^\wedge(n))=a_*(\ZZ_\ell(n))=R\Gamma(X_{et},\ZZ_\ell(n))
\]
is a dualizable object in $D^b_c({\spec k}_{et},\ZZ_\ell)$.
For $k$ separably closed, the latter category simply is the bounded derived
category of $\ZZ_\ell$-modules of finite type, and
this proves in particular that $\ell$-adic cohomology\index{l-adic@$\ell$-adic!cohomology}
\[
H^i_{et}(X,\ZZ_\ell(n))=H^i(R_\ell(M(X)^\wedge(n)))
\]
is of finite type as a $\ZZ_\ell$-module for all $i$ (and trivial for all but
finitely many $i$'s). Similarly, the $\ell$-adic realization of
$a_!(\ZZ)$ gives $\ell$-adic cohomology with compact
support\index{l-adic@$\ell$-adic!cohomology with compact support}
\[
H^i_{et,c}(X,\ZZ_\ell(n))=H^i(R_\ell(a_!(\ZZ)(n)))\ .
\]
\end{paragr}

\begin{paragr}
In particular,
for any field $k$ of characteristic prime to $\ell$,
we have a symmetric monoidal functor
\[R_{\ell}:DM_{h,lc}(k,\mathbf{Z})\to D_{c}^{b}(k,\mathbf{Z}_{\ell})\, \]
inducing the map of rings
\[R_\ell:\ZZ[1/p]\cong H^0\Hom_{\DM_{h,lc}(k,\ZZ)}(\ZZ,\ZZ)\to
H^0\Hom_{D^b_c(k,\ZZ_\ell)}(\ZZ_\ell,\ZZ_\ell)\cong\ZZ_\ell\, .\]
Therefore,
for an endomorphism $f:M\to M$ we have $Tr(f)\in \mathbf{Z}[1/p]$
sent to the $\ell$-adic number $Tr(R_\ell(f))\in\ZZ_\ell$. We thus get:
\end{paragr}
\begin{cor} The $\ell$-adic trace\index{independence of $\ell$}
$Tr(R_{\ell}(f))\in\mathbf{Z}[1/p]$ and is independent of $\ell$.
\end{cor}
\begin{rem}
If $k$ is separably closed, then $D_{c}^{b}(k,\mathbf{Z}_{\ell})$
simply is the derived category of $\ZZ_\ell$-modules of finite type, and
we have
\[Tr(R_\ell(f))=\sum_i (-1)^{i}Tr(H^iR_\ell(f):H^iR_\ell(M)\to H^iR_\ell(M))\]
where each $Tr(H^iR_\ell(f))$ can be computed in the usual way in terms of
traces of matrices. If $k$ is not separably closed,
we can always choose a separable closure $\bar k$ and observe that
pulling back along the map $Spec(\bar k)\to Spec(k)$ is
a symmetric monoidal functor which commutes with the $\ell$-adic
realization functor. This can actually be used to prove
that the Euler characteristic is always an integer
(as opposed to a rational number in $\mathbf{Z}[1/p]$): if $f=1_M$ is the
identity, the trace of $R_\ell(f)$ can be computed as an alternating sum
of ranks of $\ZZ_\ell$-modules of finite type.
\end{rem}
\begin{cor}
For any dualizable object $M$ in $\DM_h(k,\ZZ)$, we have
$\chi(M)\in\ZZ$.
\end{cor}
\begin{paragr}
Let $A$ be a ring. A function $f:X\to A$ from a topological space
to a ring is constructible\index{constructible!function} if there is a finite stratification of $X$ by
locally closed $X_{i}$ such that each $f|_{X_{i}}$ is constant.
We denote by $C(X,A)$ the ring of constructible functions with values in $A$
on $X$. For a scheme $X$, we define $C(X,A)=C(|X|,A)$, where $|X|$
denotes the topological space underlying $X$.
\end{paragr}
\begin{paragr}
Recall that, for a stable $\infty$-category $C$, we have its Grothendieck group $K_0(C)$:
the free monoid generated by isomorphism classes $[x]$ of objects $x$ of $C$,
modulo the relations $[x]=[x']+[x'']$ for each cofiber sequence
$x'\to x\to x''$. In particular, we have the relations $0=[0]$ and
$[x]+[y]=[x\oplus y]$.
This monoid turns out to be an abelian group with
$-[x]=[\Sigma(x)]$.
If ever $C$ is symmetric monoidal, then $K_0(C)$
inherits a commutative ring structure with multiplication $[x][y]=[x\otimes y]$.
\end{paragr}
\begin{paragr}
We have the Euler characteristic map
$DM_{h,lc}(X,\ZZ)\stackrel{\chi}{\to}C(X,\ZZ)$.
It is defined by $\chi(M)(x)=\chi({x}^{*}M)$,
where the point $x$ is seen as a map $x:Spec(\kappa(x))\to X$.
Recall that if $M\in DM_{h}(X,\Lambda)$ is locally constructible then
there is $U\subseteq X$ open and dense such that $M|_{(X-U)_{red}}$
is locally constructible and $M|_{U}$ is dualizable. 
Therefore, by noetherian induction, we see that
$\chi(M):|X|\to \ZZ$
is a constructible function indeed.
For any cofiber sequence of dualizable objects
\[M'\to M\to M''\,  ,\]
we have
\[\chi(M)=\chi(M')+\chi(M'')\, .\]
Since $\chi(M\otimes N)=\chi(M)\chi(N)$,
we have a morphism of rings:
\[
\chi:K_{0}(DM_{h,lc}(X,\mathbf{Z}))\to C(X,\ZZ)\, ,
\]
 and we have a commutative triangle: 
$$\begin{tikzcd}[row sep=normal, column sep=tiny]
K_0(\DMhlc(X,\ZZ))\ar{rr}{R_{\ell}}\ar{dr}[swap]{\chi}&&
K_0(D^b_c(X_{et},\ZZ_\ell))\ar{dl}{\chi}\\
&C(X,\ZZ)&
\end{tikzcd}$$
\end{paragr}
\begin{paragr}
Given a stable $\infty$-category $C$, there is the full subcategory
$C_{tors}$ which consists of objects $x$ such that there exists an integer $n$ such that
$n.1_x\cong 0$. One checks that $C_{tors}$ is a thick subcategory of $C$ and one
defines the Verdier quotient $C\otimes\QQ=C/C_{tors}$. All this is a fancy way to say that
one defines $C\otimes\QQ$ as the $\infty$-category with the same set of objects as $C$,
such that $\pi_0Map_C(x,y)\otimes\QQ=\pi_0Map_{C\otimes\QQ}(x,y)$ for all $x$ and $y$.
This is how one defines $\ell$-adic sheaves:
\[D^b_c(X_{et},\QQ_\ell)=D^b_c(X_{et},\ZZ_\ell)\otimes\QQ\, .\]
When it comes to motives, we can prove that, when $X$ is noetherian of finite
dimension, the canonical functor
\[\DM_{h,lc}(X,\ZZ)\otimes\QQ\to\DM_{h,lc}(X,\QQ)\]
is fully faithful and almost an equivalence: a Morita equivalence.
Since $\DM_{h,lc}(X,\QQ)$ is idempotent complete, that means that any
$\QQ$-linear locally constructible motive is a direct factor
of a $\ZZ$-linear
one. Furthermore, one checks that $D^b_c(X_{et},\QQ_\ell)$ is
idempotent complete (because it has a bounded $t$-structure), so that we get
a $\QQ$-linear $\ell$-adic realization functor\index{l-adic@$\ell$-adic!realization}:
\[R_\ell:\DM_{h,lc}(X,\QQ)\to D^b_c(X_{et},\QQ_\ell)\]
which is completely determined by the fact that the following square
commutes.
\[\begin{tikzcd}
\DM_{h,lc}(X,\ZZ)\ar{r}{R_\ell}\ar{d}&D^b_c(X_{et},\ZZ_\ell)\ar{d}\\
\DM_{h,lc}(X,\QQ)\ar{r}{R_\ell} &D^b_c(X_{et},\QQ_\ell)
\end{tikzcd}\]
The $\QQ$-linear $\ell$-adic realization functor commutes with the six operations
if we restrict ourselves to quasi-excellent schemes over $\ZZ[1/\ell]$;
see \cite[7.2.24]{CD4}.

We may see these realization functors as a categorified version of
cycle class maps\index{cycle class!l-adic@$\ell$-adic}. Indeed, in view of the
representability results such as Theorem~\ref{cycle class}
\index{l-adic@$\ell$-adic!cycle class}, they induce the classical
cycle class maps in $\ell$-adic cohomomology: for a field $k$ and a separated
morphism of finite type $a:X\to\spec k$, we have
\[
\begin{tikzcd}
H^0(\Hom_{DM_{h}(X,\mathbf{Q})}(\mathbf{Q}(n)[2n],a^{!}\mathbf{Q}))
\ar[equals]{d}{\wr}\ar{r}{R_\ell}&
H^0(\Hom_{D^b_c(X_{et},\QQ_\ell)}(\mathbf{Q}_\ell(n)[2n],a^!\mathbf{Q}_\ell))
\ar[equals]{d}{\wr}\\
CH_{n}(X)\otimes\mathbf{Q}\ar[dotted]{r}&
H^{2n}_{c}(X_{et},\mathbf{Q}_\ell(n))^{\wedge}
\end{tikzcd}
\]
If $X$ is regular (e.g. smooth) this gives by Poincar\'e duality the
cycle class map:
\[
CH^n(X)\to H^{2n}(X_{et},\mathbf{Q}_\ell(n))\, .
\]
One can lift these cycle class maps to integral coefficients
using similar arguments from $cdh$-motives; see~\cite{CD5}.
\end{paragr}
\begin{thm}
There is a canonical exact sequence of the form:\label{virtual}
\[K_0(\DM_{h,lc}(X,\ZZ)_{tors})\to
K_0(\DM_{h,lc}(X,\ZZ))\to
K_0(\DM_{h,lc}(X,\QQ))\to 0\, .\]
\end{thm}
\begin{proof}
Let $\DM_{h}(X,\ZZ)'$ be the smallest localizing subcategory of $\DM_{h}(X,\ZZ)$
generated by $\DM_{h,lc}(X,\ZZ)_{tors}$.
We also define $D(X_{et},\ZZ)'$ as the smallest localizing subcateory of $D(X_{et},\ZZ)$
generated by objects of the form $j_!(\mathcal{F})$, where $j:U\to X$
is a dense open immersion and $\mathcal{F}$ is bounded with constructible cohomology sheaf,
such that there is a prime $p$ with the following two properties:
\begin{itemize}
\item $p.1_\mathcal{F}=0$;
\item $p$ is invertible in $\mathcal{O}_U$.
\end{itemize}
Then a variant of the rigidity theorem \ref{rigidity theorem}
(together with remark \ref{rigidity artin}) gives an equivalence of $\infty$-categories:
\[\DM_{h}(X,\ZZ)'\cong D(X_{et},\ZZ)'\, .\]
One then checks that the $t$-structure on $D(X_{et},\ZZ)'$ induces a bounded $t$-structure on
$\DM_{h,lc}(X,\ZZ)_{tors}$ (with noetherian heart, since we get a Serre subcategory of
constructible \'etale sheaves of abelian groups on $X_{et}$).
Using the basic properties of non-connective $K$-theory \cite{ncK0,ncK,ncK2}, we see that
we have an exact sequence
\[K_0(\DM_{lc}(X)_{tors})\to K_0(\DM_{lc}(X))\to K_0(\DM_{lc}(X)_\QQ)
\to K_{-1}(\DM_{lc}(X)_{tors})\, ,\]
where $\DM_{lc}(X)=\DM_{h,lc}(X,\ZZ)$ and $\DM_{lc}(X)_\QQ=\DM_{h,lc}(X,\QQ)$.
By virtue of a theorem of Antieau, Gepner and Heller \cite{tK}, the existence of
a bounded $t$-structure with noetherian heart implies that
$K_{-i}(\DM_{h,lc}(X,\ZZ)_{tors})=0$ for all $i>0$.
\end{proof}
Here is a rather concrete consequence (since $\chi(M)=0$ for $M$ in $\DM_{h,lc}(X,\ZZ)_{tors}$).
\begin{cor}
For any $M$ in $\DM_{h,lc}(X,\QQ)$, there exists $M_0$ in $\DM_{h,lc}(X,\ZZ)$, such that,
for any point $x$ in $X$, we have $\chi(x^*M)=\chi(x^*M_0)$.
\end{cor}
\begin{rem}
It is conjectured that there is a (nice) bounded $t$-structure on $\DM_{h,lc}(X,\QQ)$.
Since $\DM_{h,lc}(X,\ZZ)_{tors})$ has a bounded $t$-structure, this would imply the existence
of a bounded $t$-structure on $\DM_{h,lc}(X,\ZZ)$, which, in turns would imply the vanishing of
$K_{-1}(\DM_{h,lc}(X,\ZZ))$ (see \cite{tK}). Such a vanishing would mean that all Verdier
quotients of $\DM_{h,lc}(X,\ZZ)$ would be idempotent-complete
(see \cite[Remark~1 p.~103]{ncK0}). In particular,
we would have an equivalence of $\infty$-categories
$\DM_{h,lc}(X,\ZZ)\otimes\QQ=\DM_{h,lc}(X,\QQ)$. The previous proposition
is a virtual approximation of this expected equivalence.
\end{rem}
\begin{paragr}
Let $R$ be a ring and let $W(R)=1+R[[t]]$ the set of power series
with coefficients in $R$ and leading term equal to $1$.
It has an abelian group
structure defined by the multiplication of power series. And it has a unique
multiplication $\ast$ such that $(1+at)\ast(1+bt)=1+abt$, turning $W(R)$ into a
commutative ring: the ring of Witt vectors.
We also have the subset $W(R)_{rat}\subseteq W(R)$
of rational functions, which one can prove to be a subring.
Given a (stable) $\infty$-category $C$, we define
\[
C^{\mathbf{N}}=\{\text{objects of \ensuremath{C} equipped with an endomorphism}\}\, .
\]
This is again a stable $\infty$-category. For $C=Perf(R)$ the $\infty$-category
of perfect complexes on the ring $R$,
we have an exact sequence 
\[
0\to K_{0}(Perf(R))\to K_{0}(Perf(R)^{\mathbf{N}}){\to}W(R)_{rat}\to 0
\]
where the first map sends a perfect complex of $R$-modules $M$ to the class of $M$ equipped
with the zero map $0:M\to M$, while the second maps sends $f:M\to M$ to $\det(1-tf)$
(it is sufficient to check that these maps are well defined when $M$ is a projective
module of finite type, since these generate the $K$-groups); see \cite{kendo}.
The first map identifies $K_0(R)$ with an ideal of $K_{0}(Perf(R)^{\mathbf{N}})$ so that
we really get an isomorphism of commutative rings:
\[K_{0}(Perf(R)^{\mathbf{N}})/K_{0}(Perf(R))\cong W(R)_{rat}\, .\]
\end{paragr}
\begin{paragr}
Let $k$ be a field with a given algebraic closure $\bar k$, as well as prime
number $\ell$ which is distinct from the characteristic of $k$.
We observe that $D^b_c(\bar k,\QQ_\ell)$ simply is the bounded derived category of
complexes of finite dimensional $\QQ_\ell$-vector spaces. We thus have a
symmetric monoidal realization functor
\[\DM_{h,lc}(k,\mathbf{Q})\to D^b_c(k,\QQ_\ell)\to D^b_c(\bar k,\QQ_\ell)\cong Perf(\QQ_\ell)\, .\]
This induces a functor
\[DM_{h,lc}(k,\mathbf{Q})^{\mathbf{N}}\to Perf(\mathbf{Q_{\ell}})^{\mathbf{N}}\, ,\]
and thus a map 
\[
K_{0}(DM_{h,lc}(k,\mathbf{Q})^{\mathbf{N}})\to K_{0}(Perf(\mathbf{Q}_{\ell})^{\mathbf{N}})
\]
inducing a ring homomorphism, the $\ell$-adic
Zeta function\index{Zeta function!$\ell$-adic}\index{l-adic@$\ell$-adic!Zeta function}
\[
Z_\ell:{K_{0}(DM_{h,lc}(k,\mathbf{Q})^{\mathbf{N}})}/{K_{0}(DM_{h,lc}(k,\mathbf{Q}))}
\to W(\mathbf{Q}_{\ell})_{rat}\subseteq1+\mathbf{Q}_{_{\ell}}[[t]]\, .
\]
On the other hand, for an endomorphism $f:M\to M$ in $\DM_{h,lc}(X,\QQ)$, one 
defines its motivic Zeta function\index{Zeta function!motivic} as follows
\[
Z(M,f)=\exp\Big(\sum_{n\ge1}Tr(f^{n})\frac{\, t^{n}}{n}\Big)\in 1+\QQ[[t]]\, .
\]
Basic linear algebra show that $Z(M,f)=Z_\ell(M,f)$ (see \cite{kendo}).
In particular, we
see that the $\ell$-adic Zeta function $Z_\ell(M,f)$ has rational coefficients
and is independent of $\ell$\index{independence of $\ell$},
while the motivic Zeta function $Z(M,f)$ is rational. In other words, we get a morphism
of rings
\[
Z:K_{0}(DM_{h,lc}(X,\mathbf{Q})^{\mathbf{N}})/K_{0}(DM_{h,lc}(X,\mathbf{Q}))
\to W(\mathbf{Q})_{rat}\subset W(\mathbf{Q})\, .
\]
Concretely, if there is a cofiber sequence of motivic sheaves equipped with endomorphisms
in the stable $\infty$-category $DM_{h,lc}(X,\mathbf{Q})^{\mathbf{N}}$
\[(M',f')\to (M,f)\to (M'',f'')\, ,\]
then
\[Z(M,f)(t)=Z(M',f')(t)\cdot Z(M'',f'')(t)\]
holds in $\mathbf{Q}[[t]]$. And for two motivic sheaves equipped with endomorphisms
$(M,f)$ and $(M',f')$ in $DM_{h,lc}(X,\mathbf{Q})^{\mathbf{N}}$, there is
\[Z(M\otimes M',f\otimes f')=Z(M,f)\ast Z(M',f')\]
where $\ast$ denotes the multiplication in the big ring of Witt vectors $W(\mathbf{Q})$.
\end{paragr}
\begin{paragr}
Take $k=\mathbf{F}_{q}$ a finite field and let
$M_{0}\in DM_{h,lc}(k,\mathbf{Q})$, with $M=p^{*}M_{0}$, $p:spec(\bar{k})\to spec(k)$. 
Let $F:M\to M$ be the induced Frobenius.
We define the \emph{Riemann-Weil Zeta function}\index{Zeta function!Riemann-Weil}
of $M_0$ as:
\[
\zeta(M_{0},s)=Z(M,F)(t),\quad t=p^{-s}.
\]
The fact that the assignment $Z(-,F)$ defines a morphism of rings
with values in $W(\QQ)$ can be used to compute explicitely the Zeta function
of many basic schemes such as $\mathbf{P}^n$ of $(\mathbf{G}_m)^r$;
see \cite[Remark 2.2]{Ramachandran} for instance.
\end{paragr}

\subsection{Grothendieck-Verdier duality}
\index{duality!Grothendieck-Verdier}\index{Grothendieck-Verdier duality}

\begin{paragr}
Take $S$ be a quasi-excellent regular scheme.
We choose a $\otimes$-invertible object $I_{S}$ in $DM_{h}(S,\Lambda)$
(e.g. $I_S=\ZZ(d)[2d]$, where $d$ is the Krull dimension of $S$).
For $a:X\to S$ separated of finite type,
we define $I_{X}=a^{!}I_{S}$. 

Define $\mathbf{D}_{X}:DM_{h}(X,\Lambda)^{\mathit{op}}\to DM_{h}(X,\Lambda)$
by 
\[
\mathbf{D}_{X}(M)=\uHom(M,I_{X}).
\]
We will sometimes write $\mathbf{D}(M)=\mathbf{D}_X(M)$.
\end{paragr}
\begin{thm}
For $M$ a locally constructible motivic sheaf over $X$,
the canonical map $M\to\mathbf{D}_{X}\mathbf{D}_{X}(M)$
is an equivalence. 
\end{thm}
There is a proof in the literature under the additional
assumption that $S$ is of finite type over
an excellent scheme of dimension $\le2$ (see \cite{CD3,CD4}). But there is
in fact a proof which avoids this extra hypothesis using higher categories.
Here is a sketch.
\begin{proof}
The formation of the Verdier dual is compatible with pulling back along
an \'etale map. We may thus assume that $M$ is constructible.
The full subcategory of those $M$'s such that the biduality map of the theorem
is invertible is thick. Therefore, we may assume that $M=M(U)$ for some
smooth $X$-scheme $U$. In particular, we may assume that $M=
\Lambda\otimes\Sigma^\infty\ZZ(U)$. It is thus sufficient to prove the
case where $\Lambda=\ZZ$. By standard arguments, we see that is is sufficient
to prove the case where $\Lambda$ is finite or $\Lambda=\QQ$.
Such duality theorem is a result of Gabber~\cite{ILO}
for the derived category of
sheaves on the small \'etale site of $X$ with coefficients in $\Lambda$
of positive characteristic with $n$
invertible in $\mathcal{O}_X$.
By Theorem \ref{rigidity theorem}
and Remark \ref{rigidity artin},
this settles the case where $\Lambda$ is finite.
It remains to prove the case where $\Lambda=\QQ$. We will first prove
the following statement.
For each separated morphism of finite type $a:X\to S$, and each integer $n$,
the natural map
\[\Hom_{\DM_h(X,\QQ)}(\QQ,\QQ(n))\to\Hom_{\DM_h(X,\QQ)}(I_X,I_X(n))\]
is invertible in $D(\QQ)$ (this is the map obtained by applying the global section
functor $\Hom(\QQ,-)$ to the unit map $\QQ\to \uHom(I_X,I_X)$).
We observe that we may see this map as a morphism of presheaves of
complexes of $\QQ$-vector spaces
\[E\to F\]
where $E(X)=\Hom_{\DM_h(X,\QQ)}(\QQ,\QQ(n))$ and $F(X)=\Hom_{\DM_h(X,\QQ)}(I_X,I_X(n))$.
For a morphism of $S$-schemes $f:X\to Y$, the induced map $E(Y)\to E(X)$ is induced by the
functor $f^*$, while the induced map $F(Y)\to F(X)$ is induced by the functor $f^!$
(and the fact that $f^!(I_Y)\cong I_X$).\footnote{This is where $\infty$-category theory
appears seriously: proving that the construction $f\mapsto f^!$ actually defines
a presheaf is a highly non-trivial homotopy coherence proplem. Such construction
is explained in \cite[Chapter 10]{robalo00},
using the general results of \cite{lw1,lw2}.}
Now, we observe that both $E$ and $F$ are in fact $h$-sheaves of complexes of $\QQ$-vector
spaces. Indeed, using \cite[Proposition 3.3.4]{CD3}, we see that $E$ and $F$ satisfy Nisnevich
excision and thus are Nisnevich sheaves. On the other hand,
one can also characterise $h$-descent for $\QQ$-linear Nisnevich sheaves
by suitable excision properties \cite[Theorem 3.3.24]{CD3}.
Such properties for $E$ and $F$ follow right away from
\cite[Theorem 14.3.7 and Remark 14.3.38]{CD3}, which
proves the property of $h$-descent for $E$ and $F$.
By virtue of Theorem \ref{resolution of singularities}, it is sufficient to prove that
$E(X)\cong F(X)$ for $X$ regular and affine. In particular, $a:X\to S$ factors
through a closed immersion $i:X\to \mathbf{A}^n\times S$. By relative purity,
we have
\[I_{\mathbf{A}^n\times S}\cong p^*(I_S)(n)[2n]\]
and thus $I_{\mathbf{A}^n\times S}$ is $\otimes$-invertible (where $p:\mathbf{A}^n\times S\to S$
is the second projection). This implies that
\[I_X\cong i^!(I_{\mathbf{A}^n\times S})\cong i^!(\QQ)\otimes i^*(I_{\mathbf{A}^n\times S})\]
(Hint: use the fact that $i^!\uHom(A,B)\cong\uHom(i^*A,i^!B)$). By Absolute Purity, we have
$i^!\QQ\cong\QQ(-c)[-2c]$, where $c$ is the codimenion of $i$.
In particular, the object $I_X$ is $\otimes$-invertible, and thus the unit map
$\QQ\to\uHom(I_X,I_X)$ is invertible. This implies that the map $E(X)\to F(X)$ is invertible
as well.

We will now prove that the unit map
\[\QQ\to\uHom(I_X,I_X)\]
is invertible in $\DM_h(X,\QQ)$ for any separated $S$-scheme of finite type $X$.
Equivalently, we have to prove that, for any smooth $X$-scheme $U$ and
any integer $n$, the induced map
\[\Hom_{\DM_h(X,\QQ)}(M(U),\QQ(n))\to\Hom_{\DM_h(X,\QQ)}(M(U),\uHom(I_X,I_X)(n))\]
is invertible in $D(\QQ)$. 
But we have
\[\Hom(f_\sharp\QQ,\QQ(n))\cong\Hom(\QQ,\QQ(n))=E(U)\]
with a smooth structural map $f:U\to X$, and
\begin{align*}
\Hom(f_\sharp\QQ,\uHom(I_X,I_X)(n))
&\cong\Hom(\QQ,f^*\uHom(I_X,I_X)(n))\\
&\cong\Hom(\QQ,\uHom(f^*I_X,f^*I_X)(n))\\
&\cong\Hom(\QQ,\uHom(f^!I_X(-d)[-2d],f^!I_X(-d)[-2d])(n))\\
&\cong\Hom(\QQ,\uHom(f^!I_X,f^!I_X)(n))\\
&\cong\Hom(\QQ,\uHom(I_U,I_U)(n))=F(U)\, .
\end{align*}
In other words, we just have to check that the map $E(U)\to F(U)$ is invertible,
which we already know.

Finally, we can prove that the canonical map $M\to\mathbf{D}_{X}\mathbf{D}_{X}(M)$
is invertible. As already explained at the beginning of the proof, it is sufficient
to prove this when $M$ is constructible. By virtue of Proposition \ref{proper devissage},
it is sufficient to prove the case where $M=f_*(\QQ)$, for $f:Y\to X$
a proper map. We have:
\begin{align*}
\mathbf{D}_Xf_*\QQ
&=\uHom(f_*\QQ,I_X)\\
&\cong f_*\uHom(\QQ,f^!I_X)\\
&\cong f_*f^!I_X\\
&\cong f_*I_Y\, .
\end{align*}
Therefore, we have
\begin{align*}
\mathbf{D}_X\mathbf{D}_X(M)
&\cong\mathbf{D}_X f_*I_Y\\
&\cong\uHom(f_*I_Y,I_X)\\
&\cong f_*\uHom(I_Y,f^!I_X)\\
&\cong f_*\uHom(I_Y,I_Y)\\
&\cong f_*\QQ=M\, ,
\end{align*}
and this ends the proof.
\end{proof}

\begin{cor}
For locally constructible motives and $f$
a morphism between separated $S$-schemes of finite type,
we have:
\begin{align*}
\mathbf{D}f_{*}&\cong f_{!}\mathbf{D}\\
\mathbf{D}f_{!}&\cong f_{*}\mathbf{D}\\
\mathbf{D}f^{!}&\cong f^{*}\mathbf{D}\\
\mathbf{D}f^{*}&\cong f^{!}\mathbf{D}\, .
\end{align*}
\end{cor}
(The proof is by showing tautologically the second one and the fourth one, and then deduce the
other two using that $\mathbf{D}$ is an involution.)

\begin{prop}
For any $M$ and $N$ in $\DM_h(X,\Lambda)$,
if $N$ is locally constructible, then\label{internal Hom and duality}
\[
\mathbf{D}(M\otimes\mathbf{D}N)\cong \uHom(M,N)\, .
\]
\end{prop}
\begin{proof}
We construct a canonical comparison morphism:
\[
\uHom(M,N)\to \mathbf{D}(M\otimes\mathbf{D}N)\, .
\]
By transposition, it corresponds to a map
\[M\otimes\uHom(M,N)\otimes\mathbf{D}(N)\to I_X\, .\]
Such a map is induced by the evaluation maps
\[M\otimes\uHom(M,N)\to N\quad\text{and}\quad N\otimes\mathbf{D}(N)\to I_X\, .\]
For $N$ fixed, the class of $M$'s such that this map is invertible is closed under
colimits. Therefore, we reduce the question to the case where
$M=f_{\sharp}\Lambda$ for $f:X\to S$ a smooth map of dimension $d$. In that case,
we have
\[
\uHom(M,N)\cong f_{*}f^{*}(N),
\]
while
\begin{align*}
\mathbf{D}(M\otimes\mathbf{D} N)
&\cong \mathbf{D}(f_! f^*(\Lambda(-d)[-2d])\otimes\mathbf{D} N)\\
&\cong \mathbf{D}(f_! f^*(\Lambda)\otimes\mathbf{D} N)(d)[2d]\\
&\cong \mathbf{D}(f_!f^*(\mathbf{D} N))(d)[2d]\\
&\cong f_* f^!(\mathbf{D}\mathbf{D}N)(d)[2d]\\
&\cong f_* f^* N\, ,
\end{align*}
which ends the proof.
\end{proof}
\begin{cor}
For $M$ and $N$ locally constructible on $X$, we have:
\[ M\otimes N\cong\mathbf{D}\uHom(M,\mathbf{D}N)   \, .\]
\end{cor}
\subsection{Generic base change: a motivic variation on Deligne's proof}
\begin{paragr}
The following statement,
is a motivic analogue of Deligne's generic base change theorem
for torsion \'etale sheaves \cite[Th.~Finitude, 1.9]{SGA4demi}.
The proof follows essentially the same pattern as Deligne's original argument,
except that locally constant sheaves are replaced by dualizable
objects, as we will explain below. We will write
$\DM_h(X)=\DM_h(X,\Lambda)$ for some fixed choice of coefficient ring $\Lambda$.
\end{paragr}
\begin{thm}[Motivic generic base change formula]\index{base change formula!generic}
Let $f:X\to Y$ be a morphism
between separated schemes of finite type over a noetherian base scheme $S$.
Let $M$ be a locally constructible $h$-motive on $X$. Then there
is a dense subscheme $U\subset S$ such that the formation of $f_{*}(M)$
is compatible with any base-change which factors through $U$.
Namely, for each $w:S'\to S$ factoring through $U$ we have 
\[
v^{*}f_{*}M\cong f'_{*}u^{*}M
\]

where 
\[
\begin{tikzcd}
X'\ar{r}{u}\ar{d}{f'} & X\ar{d}{f}\\
Y'\ar{r}{v}\ar{d} & Y\ar{d}\\
S'\ar{r}{w} & S
\end{tikzcd}
\]

is the associated pull-back diagram. \label{thm:gbc}
\end{thm}

\begin{rem}
The motivic generic base change formula is also a kind of
independence of $\ell$ result\index{independence of $\ell$}
for each prime $\ell$ so that
the $\ell$-adic realization is defined, the formation
of $f_*R_\ell(M)\cong R_\ell(f_*M)$ is compatible
with any base change over $U\subset S$, where $U$
is a dense open subscheme which is given independently of $\ell$.
\end{rem}
The first step in the proof of Theorem \ref{thm:gbc} is to find
sufficient conditions for the formation a direct image to be compatible
with arbitrary base change.

\begin{prop}
Let $f:X\to S$ be a smooth morphism of finite type between noetherian schemes, and
let us consider a locally constructible $\h$-motive $M$ over $X$. Assume that
$M$ is dualizable in $\DMhlc(X)$ and that the direct image with compact support of its dual
$f_!(M^\wedge)$ is dualizable as well in $\DMhlc(S)$.
Then $f_*(M)$ is dualizable (in particular, locally constructible), and,
for any pullback square of the form
$$\begin{tikzcd}
X'\ar{r}{u}\ar{d}[swap]{f'}&X\ar{d}{f}\\
S'\ar{r}{v}&S
\end{tikzcd}$$
the morphism $f'$ is smooth, the pullback $u^*(M)$ is dualizable, so is $f'_!(u^*(M)^\wedge)$, and, furthermore,
the canonical base change map
$v^*f_*(M)\to f'_*u^*(M)$ is invertible.\label{prop:gbc}
\end{prop}
\begin{proof}
If $d$ denotes the relative dimension of $X$ over $S$
(seen as a locally constant function over $S$), we have:
\begin{align*}
f_*(M)&\simeq f_*\sHom(M^\wedge,\Lambda)\\
&\simeq f_*\sHom(M^\wedge,f^!\Lambda)(-d)[-2d]\\
&\simeq \sHom(f_!(M^\wedge),\Lambda)(-d)[-2d]\\
&\simeq (f_!(M^\wedge))^\wedge(-d)[-2d]
\end{align*}
(where the dual of a dualizable object $A$ is denoted by $A^\wedge$).
Remark that pullback functors $v^*$ are symmetric monoidal and thus
preserve dualizable objects as well as the formation of their duals.
Therefore, for any pullback square of the form
$$\begin{tikzcd}
X'\ar{r}{u}\ar{d}[swap]{f'}&X\ar{d}{f}\\
S'\ar{r}{v}&S
\end{tikzcd}$$
we have that $f'$ is smooth of relative dimension $d$, that $u^*(M)$ is dualizable with dual
$u^*(M)^\wedge\simeq u^*(M^\wedge)$, and:
\begin{align*}
v^*f_*(M)&\simeq v^*(f_!(M^\wedge))^\wedge(-d)[-2d]\\
&\simeq (v^*f_!(M^\wedge))^\wedge(-d)[-2d]\\
&\simeq (f'_!u^*(M^\wedge))^\wedge(-d)[-2d]\\
&\simeq (f'_!(u^*(M)^\wedge))^\wedge(-d)[-2d]
\end{align*}
This also shows that $f'_!(u^*(M)^\wedge)$ is dualizable and thus that there is a canonical isomorphism
$$(f'_!(u^*(M)^\wedge))^\wedge(-d)[-2d]\simeq f'_*(u^*(M))\, .$$
We deduce right away from there that the canonical base change map
$v^*f_*(M)\to f'_*(u^*(M))$ is invertible.
\end{proof}
\begin{rem}
In the preceding proposition, we did not use any particular property of $\DMhlc$:
the statement and its proof hold in any context in which we have the six operations
(more precisely, we mainly used the relative purity theorem as well as the proper base change theorem).
\end{rem}

In order to prove Theorem \ref{thm:gbc} in general, we need to verify the following property
of $\h$-motives.

\begin{prop}
Let $S$ be a noetherian scheme of finite dimension, and $f:Y\to S$ a quasi-finite morphism of finite type.
The functors $f_!:\DMh(X)\to\DMh(S)$ and $f_*:\DMh(X)\to\DMh(S)$
are conservative.\label{prop:localconservativity}
\end{prop}
\begin{proof}
If $f$ is an immersion, then $f_!$ and $f_*$ are fully faithful, hence conservative.
Since the composition of two conservative functors is conservative,
Zariski's Main Theorem implies that it is sufficient to prove the case where $f$ is finite.
In this case,
since the formation of $f_!\simeq f_*$ commutes with base change along any map $S'\to S$,
by noetherian induction, it is sufficient to prove this assertion after restricting to a dense
open subscheme of $S$ of our choice.
Since, for $\h$-motives, pulling back along a surjective \'etale morphism is conservative,
we may even replace $S$ by an \'etale neighbourhood of its generic points.
For $f$ surjective and radicial, \cite[Proposition~6.3.16]{CD4} ensures that $f_!$ is an equivalence
of categories. We may thus assume that $f$ also is \'etale.
If ever $X=X'\amalg X''$, and if $f'$ and $f''$ are the restriction of $f$ to $X'$
and $X''$, respectively, then we have
$\DMh(X)\simeq \DMh(X')\times\DMh(X'')$, and the functor $f_!$ decomposes into
$$f_!(M)=f'_!(M')\oplus f''_!(M'')$$
for $M=(M',M'')$. Therefore, it is then sufficient to prove the proposition for $f'$ and $f''$
separately. Replacing $S$ by an \'etale neighbourhood of its generic points, we may thus assume
that either $X$ is empty, either $f$ is an isomorphism, in which cases the assertion is trivial.
\end{proof}

\begin{paragr}
Let $P(n)$ be the assertion that, whenever $S$ is integral
and $f:X\to Y$ is a separated morphism of $S$-schemes
of finite type, such that the dimension of the generic fiber of $X$ over $S$
is smaller than or equal to $n$, then, for any locally construtible $\h$-motive $M$
on $X$, there is a dense open subscheme $U$ of $S$ such that the formation
of $f_*(M)$ is compatible with base change along maps
$S'\to U\subset S$.

From now on, we fix a separated morphism of $S$-schemes of finite type $f:X\to Y$;
as well as a locally constructible $\h$-motive $M$ on $X$.
\end{paragr}

\begin{Lemma}
The property that there exists a dense open subscheme $U\subset S$
such that the formation of $f_*(M)$ is stable under any base change along
maps $S'\to U\subset S$ is local on $Y$ for the Zariski
topology.\label{lemma:gbc is local}
\end{Lemma}

\begin{proof}
Indeed, assume that there is an open covering $Y=\bigcup_j V_i$
such that, for each $j$, there is a dense open subset $U_j\subset U$
with the property that the formation of the motive
$(f^{-1}(V_j)\to V_j)_*(M_{f^{-1}(V_j)})$ is stable under any base
change along maps of the form $S'\to U_j\subset S$.
Since $Y$ is noetherian, we may assume that there finitely many $V_j$'s,
so that $U=\bigcap_j U_j$ is a dense open subscheme of $S$.
For any $j$, the formation of
$(f^{-1}(V_j)\to V_j)_*(M_{f^{-1}(V_j)})$ is stable under any base
change along maps of the form $S'\to U\subset S$.
Since pulling back along open immersions commutes with any push-forward,
one deduces easily that the formation of $f_*(M)$
is stable under any base change of the form
$(f^{-1}(V_j)\to V_j)_*(M_{f^{-1}(V_j)})$ is stable under any base
change along maps of the form $S'\to U\subset S$.
\end{proof}

\begin{Lemma}
Assume that there is a compactification of $Y$: an
open immersion $j:Y\to \bar Y$ with $\bar Y$ a proper $S$-scheme.
If there is a dense open subscheme $U$ such that the formation
of $(jf)_*(M)$ is compatible with all base changes
along maps $S'\to U\subset S$, then 
the formation
of $f_*(M)$ is compatible with all base changes
along maps $S'\to U\subset S$.\label{lemma:gbc compatible with compacticications}
\end{Lemma}

\begin{proof}
This follows right away from the fact that
pulling back along $j$ is compatible with any base changes
and from the fully faithfulness of the functor $j_*$
(so that $j^*j_* f_*(M)\simeq f_*(M)$).
\end{proof}

\begin{Lemma}
Assume that $S$ is integral, that the dimension of the generic
fiber of $X$ over $S$ is $n\geq 0$, and that $P(n-1)$ holds.
If $X$ is smooth over $S$, and if $M$ is dualizable, then there is a dense open
subscheme of $S$ such that the formation of $f_*(M)$ is stable under
base change along maps $S'\to U\subset S$.\label{lemma: gbc with X smooth}
\end{Lemma}

\begin{proof}
Since pulling back along open immersions commutes with any push-forward,
and since $Y$ is quasi-compact, the problem is local over $Y$.
Therefore, we may assume that $Y$ is affine.
Let us choose a closed embedding $Y\subset\mathbf{A}^d_S$
determined by $d$ functions $g_i:Y\to\mathbf{A}^1_S$, $1\leq i\leq d$.
For each index $i$,
we may apply $P(n-1)$ to $f$, seen as open
embedding of schemes over $\mathbf{A}^1_S$ through the structural map $g_i$.
This provides a dense open subscheme $U_i$ in $\mathbf{A}^1_S$ such that
the formation of $f_*(M)$ is compatible with any base change of $g_i$ along
a map $S'\to\mathbf{A}^1_S$ which factors through $U_i$.
Let $V$ be the union of all the open subschemes $g^{-1}_i(U_i)$, $1\leq i\leq d$,
and let us write $j:V\to Y$ for the corresponding open immersion.
Then the formation of $j_!j^*f_*(M)$ is compatible
with any base change $S'\to S$.
Let us choose a closed complement $i:T\to Y$ to $j$.
Then $T$ is finite: the reduced geometric fibers of $T/S$ are traces on $Y$
of the subvarieties of $\mathbf{A}^d$ determined by the vanishing of all the non constant polynomials
$p_i(x_i)=0$, $1\leq i\leq d$, where $p_i(x)$ is a polynomial such that $U_i=\{p_i(x)\neq 0\}$.

We may now consider the closure $\bar Y$ of $Y$ in $\mathbf{P}^n_S$.
Any complement of $V$ in $\bar Y$ is also finite over a dense open
subscheme of $S$: the image in $S$ of the complement of $V$ in $\bar V$
is closed (since $\bar V$ is proper over $S$), and does not contain the generic
point (since the generic fiber of $X$ is not empty), so that we may
replace $S$ by the complement of this image.
By virtue of Lemma \ref{lemma:gbc compatible with compacticications},
we may replace $Y$ by $\bar Y$, so that we are reduced to the following situation:
the scheme $Y$ is proper over $S$, and there is a dense open immersion
$j:V\to Y$ with the property that the formation of $j_!j^*f_*(M)$ is compatible
with any base change $S'\to S$, and that after shrinking $S$,
there is a closed complement $t:T\to Y$ of $V$ which is finite over $S$.
We thus have the following canonical cofiber sequence
$$j_!j^*f_*(M)\to f_*(M)\to i_*i^*f_*(M)$$
Let $p:Y\to S$ be the structural map (which is now proper).
We already know that the formation of
$j_!j^*f_*(M)$ is compatible with any base change of the form $S'\to S$.
Therefore, it is sufficient to prove that, possibly after shrinking $S$,
the formation of $i_*i^*f_*(M)$ has the same property. Since $i_!\simeq i_*$,
this means that this is equivalent to the property that, possibly after shrinking $S$,
the formation of $i^*f_*(M)$ is compatible with any base change of the form $S'\to S$.
But the composed morphism $pi$ being finite, by virtue of Proposition \ref{prop:localconservativity},
we are reduced to prove this property for $p_*i_*i^*f_*(M)$. We then have the
following canonical cofiber sequence
$$p_*j_!j^*f_*(M)\to (pf)_*(M)\to (pi)_*i^*f_*(M)$$
By virtue of Proposition \ref{prop:gbc}, possibly after shrinking $S$,
the formation of $(pf)_*(M)$ is compatible with any base change.
Since $p$ is proper, we have the proper base change formula (because $p_!\simeq p_*$),
and therefore, the formation of $j_!j^*f_*(M)$
being compatible with any base change of the form $S'\to S$,
the formation of $p_*j_!j^*f_*(M)$ is also compatible with
any base change $S'\to S$. One deduces that, possibly after shrinking $S$
the fomration of $(pi)_*i^*f_*(M)$ is also compatible with
any base change $S'\to S$.
\end{proof}
\begin{proof} {\itshape \textbf{of Theorem \ref{thm:gbc}}}
We observe easily that it is sufficient to prove the case
where $S$ is integral.
We shall prove $P(n)$ by induction. The case $n=-1$ is clear.
We may thus assume that $n\geq 0$ and that $P(n-1)$ holds true.
Locally for the $\h$-topology, radicial surjective and integral morphisms
are isomorphisms; in particular,
pulling back along a radicial surjective and integral morphism
is an equivalence of categories which commutes with the six operations.
There is a dense open subscheme $U$ of $S$
and a finite radicial and surjective map  $U'\to U$, so that
$X'=X\times_S U'$ has a dense open subscheme
which is smooth over $U'$
(it is sufficent to prove this over the spectrum of the field of
functions of $S$, by standard limit arguments). Replacing $S$ by $U'$
and $X$ by $X'$, we may thus assume, without loss of generality,
that the smooth locus of $X$ over $S$ is a dense open subscheme.

Let $j:V\to X$ be a dense open immersion such that $V$ is smooth over $S$.
Shrinking $V$, we may assume furthermore that $M_{|V}$ is dualizable in $\DMh(V)$.
We choose a closed complement $i:Z\to X$ of $V$. With $N=i^!(M)$,
we then have the following canonical cofiber sequence:
$$i_*(N)\to M\to j_*j^*(M)$$
By virtue of Lemma \ref{lemma: gbc with X smooth}, possibly after shrinking $S$,
we may assume that the formation of $j_*(M)$ is compatible with base changes
along maps $S'\to S$. So is the formation of $i_*(N)$, since $i$
is proper. Applying the functor $f_*$ to the distinguished triangle above,
we obtain the following cofiber sequence:
$$(fi)_*(N)\to f_*(M)\to (fj)_*j^*(M)\, .$$
We may apply Lemma \ref{lemma: gbc with X smooth} to $fj$ and $M$,
and observe that $P(n-1)$ applies to $fi$ and $N$.
Therefore, there exists a dense open subscheme $U\subset S$
such that the formation of $(fi)_*(N)$ and of $(fj)_*j^*(M)$
is compatible with any base change along maps $S'\to U\subset S$.
This implies that the formation of $f_*(M)$ is compatible
with such base changes as well.
\end{proof}
\section{Characteristic classes }

\subsection{K\"unneth Formula}
\begin{paragr}
Let $k$ be a field. All schemes will be assumed to be
separated of finite type over $k$. 
\end{paragr}
\begin{thm}
Let $f:X\to Y$ be a map of schemes, and $T$ a scheme. Consider the square

\[
\begin{tikzcd}
T\times X\ar{r}{pr_2}\ar{d}{1\times f} & X\ar{d}{f}\\
T\times Y\ar{r}{pr_2}& Y
\end{tikzcd}
\]

obtained by multiplying $f:X\to Y$ and $T\to Spec(k)$.
Then $pr_2^{*}f_{*}\cong(1\times f)_{*}pr_2^{*}$ holds. 
\end{thm}
\begin{proof}
Since, for a field $k$ and $S=Spec(k)$, the only dense open
subscheme of $S$ is $S$ itself, the generic base change formula gives
that the canonical map $pr_2^{*}f_{*}(M)\to(1\times f)_{*}pr_2^{*}(M)$
is an isomorphism for any locally constructible motive $M$ on $X$.
Since we are comparing colimit preserving functors and since any
motive is a colimit of locally constructible ones, this proves
the theorem.
\end{proof}
Some consequences:
\begin{enumerate}
\item Take $X,T$ to be schemes and $pr_{2}:T\times X\to X$ the projection.
Then, for any $M$ locally constructible on $X$ we have:
\[pr_{2}^{*}\uHom(M,N)\cong \uHom(pr_{2}^{*}M,pr_{2}^{*}N)\, .\]
It is proved by producing a canonical map and then prove for a fixed
$N$ and reduce to the case where $M$ is a generator, namely $M=f_{\sharp}\Lambda$
for smooth $f$. Then we get $\uHom(M,N)\cong f_{*}f^{*}M$. 
\item For a morphism $f:X\to Y$ consider the square below.
\[
\begin{tikzcd}
T\times X\ar{r}{pr_{2}}\ar{d}{1\times f} & X\ar{d}{f}\\
T\times Y\ar{r}{pr_{2}} & Y
\end{tikzcd}
\]
Then $pr_{2}^{*}f^{!}\cong(1\times f)^{!}pr_{2}^{*}$. 
\end{enumerate}
For the proof observe that this is a local problem so that we may assume
$f$ is quasi-projective. The map $f$ then has a factorization $f=g\circ i\circ j$
where $g$ is smooth, $i$ is a closed immersion, and $j$ is an open immersion.
Then $j^{*}=j^{!}$ and $g^{*}=g^{!}(-d)[-2d]$ so we reduce to the
case where $f$ is a closed immersion. Then $f_{*}$
and $(1\times f)_{*}$ are fully faithfull hence conservative. Therefore, it
suffices to show 
\[
(1\times f)_{*}pr_{2}^{*}f^{!}\cong(1\times f)_{*}(1\times f)^{!}pr_{2}^{*}\, .
\]
Since left hand side is isomorphic to
\[
pr_{2}^{*}f_{*}f^{!}\, ,
\]
we only need to commute $f_{*}f^{!}$ and $pr_{2}^{*}$. Now
observe that $f_{*}f^{!}(M)\cong \uHom(f_{*}\Lambda,M).$ So we deduce the
commutation of $f_{*}f^{!}$ from the commutation with internal $\uHom$
and $f_{*}$ (which we both know). 
We finally have proper basechange $pr_{2}^{*}f_{*}(\Lambda)\cong(1\times f)_{*}pr_{2}^{*}$
and this finishes the proof.
\begin{rem}
If $f$ is smooth or $M$ is `smooth' (dualizable) then for all
$N$ we have 
\[
f^{*}\uHom(M,N)\cong \uHom(f^{*}M,f^{*}N)
\]
(see Remark~\ref{rem:dualizable functorial}).
\end{rem}
\begin{paragr}
For $X$ a scheme and $a:X\to Spec(k)$ we define the dualizing sheaf
to be $I_{X}=a^{!}\Lambda$ and $\mathbf{D}_{X}=\uHom(-,I_{X})$. If
$X,Y$ are schemes we can consider their product $X\times Y$ with
projections $p_{X}:X\times Y\to X$ and $p_{Y}:X\times Y\to Y$.
If $M,N$ are motivic sheaves on
$X,Y$ respectively, we can define 
\[
M\boxtimes N:=p_{X}^{*}M\otimes p_{Y}^{*}N
\]
 and then, recalling that $A\otimes B\cong\mathbf{D}\uHom(A,\mathbf{D}B)$,
we get that 
\[
M\boxtimes N\cong\mathbf{D}(\uHom(p_{X}^{*}M,\mathbf{D}p_{Y}^{*}N))
\cong\mathbf{D}\uHom(p_{X}^{*}M,p_{Y}^{!}\mathbf{D}N)
\]
 and therefore 
\[
M\boxtimes\mathbf{D}N\cong\mathbf{D}\uHom(p_{X}^{*}M,p_{Y}^{!}N)
\]
 \end{paragr}
\begin{thm}
Let $X,Y$ be schemes and $N$ locally constructible on $Y$. Then
$p_{Y}^{!}N\cong I_{X}\boxtimes N$. 
\end{thm}
\begin{proof}
Let $a_{X}$ and $a_{Y}$ be the structure maps of $X,Y$ to $Spec(k)$.
Then 
\[
p_{X}^{*}a_{X}^{!}\cong p_{Y}^{!}a_{Y}^{*}.
\]
We have $I_{X}=a_{X}^{!}(\Lambda)$ and $p_{X}^{*}(I_{X})\cong p_{X}^{!}(\Lambda).$
Moreover:
\[p_{X}^{*}(I_{X})\cong p_{X}^{*}(\mathbf{D}_{X}\Lambda)
\cong\mathbf{D}_{X\times Y}p_{X}^{!}\Lambda\cong\mathbf{D}_{X\times Y}p_{Y}^{*}(I_{Y})\, .\]
Then we have 
\begin{align*}
p_{X}^{*}I_{X}\otimes p_{Y}^{*}N
&\cong\mathbf{D}p_{Y}^{*}I_{Y}\otimes p_{Y}^{*}N\\
&\cong\mathbf{D}\uHom(p_{Y}^{*}N,p_{Y}^{*}I_{Y})\\
&\cong\mathbf{D}p_{Y}^{*}\uHom(N,I_{Y})\\
&\cong\mathbf{D}p_{Y}^{*}\mathbf{D}N\\
&\cong p_{Y}^{!}N
\end{align*}
Hence $I_{X}\boxtimes N\cong p_{Y}^{!}N$.
\end{proof}
\begin{cor}
$I_{X}\boxtimes I_{Y}\cong I_{X\times Y}$.
\end{cor}
\begin{proof}
$I_{X\times Y}\cong p_{Y}^{!}a_{Y}^{!}\Lambda\cong p_{Y}^{!}I_{Y}\cong I_{X}\boxtimes I_{Y}$. 
\end{proof}
\begin{prop}[K\"unneth Formula with compact support]
\index{K\"unneth formula!for cohomology with compact support}
Let $f:U\to X$ and $g:V\to Y$ and let $A\in DM_{h}(U,\Lambda)$ and $B\in DM_{h}(V,\Lambda)$
then 
\[
f_{!}(M)\boxtimes g_{!}(N)\cong(f\times g)_{!}(M\boxtimes N).
\]
\end{prop}
\begin{proof}
Since $(f\times g)_!\cong(f\times 1)_!(1\times g)_!$, we see that
it is sufficient to prove this when $f$ or $g$ is the identity.
Using the functorialities induced by permuting the factors
$X\times Y\cong Y\times X$, we see that it is sufficient to prove the case where
$g$ is the identity. We then have a Cartesian square
\[
\begin{tikzcd}
U\times Y\ar{r}{p_U}\ar{d}[swap]{f\times 1}&U\ar{d}{f}\\
X\times Y\ar{r}{p_X}&X
\end{tikzcd}
\]
inducing an isomorphism
\[(f\times 1)_!p_U^*\cong p_X^*f_!\, .\]
The projection formula also gives
\[(f\times 1)_!(p^*_U(M))\otimes p_Y^*(N)
\cong(f\times 1)_!(M\boxtimes N)\]
so that we get
$f_{!}(M)\boxtimes N\cong(f\times 1)_{!}(M\boxtimes N)$.
\end{proof}
\begin{cor}
For $X=Y$ we get $f_{!}(M)\otimes g_{!}(N)\cong\pi_{!}i^{*}(M\boxtimes N)$
where $\pi:U\times_{X}V\to X$ is the canonical map,
while $i:U\times_{X}V\to U\times V$ is the inclusion map. 
\end{cor}
\begin{rem}
For $f,g$ proper we get $f_{*}M\boxtimes f_{*}N\cong(f\times g)_{*}(M\boxtimes N)$. 
\end{rem}
\begin{thm}
For $M\in DM_{h,lc}(X,\Lambda)$ and $N\in DM_{h,lc}(Y,\Lambda)$
we have\label{abstract kunneth}
\[\mathbf{D}(M\boxtimes N)\cong
\mathbf{D}M\boxtimes\mathbf{D}N\, .\]
\end{thm}
\begin{proof}
We may assume that $M=f_{*}\Lambda$ and $N=g_{*}\Lambda$ with $f,g$
proper. Then 
\begin{align*}
\mathbf{D}f_{*}\Lambda\boxtimes\mathbf{D}g_{*}\Lambda
&\cong f_{*}I_{U}\boxtimes g_{*}I_{V}\\
&\cong(f\times g)_{*}(I_{U}\boxtimes I_{V})\\
&\cong(f\times g)_{*}I_{U\times V}\\
&\cong\mathbf{D}((f\times g)_{*}\Lambda)\\
&\cong\mathbf{D}(f_{*}\Lambda\boxtimes g_{*}\Lambda)\\
&\cong\mathbf{D}(M\boxtimes N).
\end{align*}
Hence $\mathbf{D}(M\boxtimes N)\cong\mathbf{D}M\boxtimes\mathbf{D}N$.
\end{proof}
\begin{cor}
$\mathbf{D}M\boxtimes N\cong \uHom(p_{X}^{*}M,p_{Y}^{!}N)$ for $M$ and $N$
locally constructible.
\end{cor}
\begin{cor}[K\"unneth Formula in cohomology]
\index{K\"unneth formula!for ordinary cohomology}
Let us consider $f:U\to X$ and $g:V\to Y$ together with $M\in DM_{h}(U,\Lambda)$
and $N\in DM_{h}(V,\Lambda)$. Then 
\[
f_{*}(M)\boxtimes g_{*}(N)\cong(f\times g)_{*}(M\boxtimes N).
\]
\end{cor}
\begin{proof}
Functors of the form $p_*$, for $p$ separated of finite type,
commute with small colimits: since they are exact, it is sufficient to
prove that they commute with small sums, which follows from
\cite[Prop.~5.5.10]{CD4}. Therefore it is sufficient to prove this when $M$
and $N$ are (locally) constructible. In this case, the series of isomorphisms
\begin{align*}
f_*(M)\boxtimes g_*(N)
&\cong\mathbf{D}\mathbf{D}(f_*(M)\boxtimes g_*(N))\\
&\cong\mathbf{D}(\mathbf{D}f_*(M)\boxtimes \mathbf{D}g_*(N))\\
&\cong\mathbf{D}(f_!\mathbf{D}M\boxtimes g_!\mathbf{D}N)\\
&\cong\mathbf{D}((f\times g)_!(\mathbf{D}M\boxtimes\mathbf{D}N))\\
&\cong\mathbf{D}((f\times g)_!(\mathbf{D}(M\boxtimes N))\\
&\cong\mathbf{D}\mathbf{D}((f\times g)_*(M\boxtimes N))\\
&\cong(f\times g)_*(M\boxtimes N)
\end{align*}
proves the claim.
\end{proof}
\begin{rem}
In the situation of the previous corollary, if $X=Y= Spec(k)$,
then also $X\times Y= Spec (k)$, so that the exterior tensor product $\boxtimes$
in $\DM_h(X\times Y,\Lambda)$ simply corresponds to the usual tensor product $\otimes$
on $\DM_h(k,\Lambda)$. We thus get a K\"unneth formula of the form
\[(a_U)_*(M)\otimes (a_V)_*(N)\cong(a_U\times a_V)_*(M\boxtimes N)\, .\]
\end{rem}
\begin{cor}
Let us consider $f:U\to X$ and $g:V\to Y$, together with $M\in DM_{h,lc}(X,\Lambda)$
and $N\in DM_{h,lc}(Y,\Lambda)$. Then\label{encore kunneth}
\[
f^{!}(M)\boxtimes g^{!}(N)\cong(f\times g)^{!}(M\boxtimes N).
\]
\end{cor}
\begin{proof}
For any separated morphism of finite $a$, the functor $a^!$ commutes with small
colimits (since, they are exact, it is sufficient to prove that they commutes with
small sums, which is asserted by \cite[Cor.~5.5.14]{CD4}).
It is thus sufficient to prove this
formula for constructible motivic sheaves.
Using the fact that the Verdier duality functor $\mathbf{D}$ exchanges
$*$'s and $!$'s as well
as Theorem \ref{abstract kunneth}, we see that it is sufficient to prove
the analogous formula obtained by considering functors of the form $(f\times g)^*$
and $f^*$, $g^*$, which is obvious.
\end{proof}
\begin{cor}
Let $X$ be a scheme together with $M,N\in\DM_{h,lc}(X,\Lambda)$. If we
denote by $\Delta:X\to X\times X$ the diagonal map, then\label{fancy hom}
\[\Delta^!(\mathbf{D}M\boxtimes N)\cong\uHom(M,N)\, .\]
\end{cor}
We have indeed:
\begin{align*}
\Delta^!(\mathbf{D}M\boxtimes N)
&\cong\mathbf{D}\Delta^*\mathbf{D}(\mathbf{D}M\boxtimes N)\\
&\cong\mathbf{D}\Delta^*(\mathbf{D}\mathbf{D}M\boxtimes\mathbf{D}N)\\
&\cong\mathbf{D}(M\otimes\mathbf{D}N)\\
&\cong\uHom(M,N)\, .
\end{align*}
\subsection{Grothendieck-Lefschetz Formula. }
As in the previous paragraph, we assume that a ground field $k$ is given, and all schemes are assumed to be separated of finite type over $k$.
\begin{defn}
Let $X$ and $Y$ be schemes, together with $M\in\DM_{h,lc}(X,\Lambda)$
and $N\in \DM_{h,lc}(Y,\Lambda)$.
A \emph{cohomological correspondence}\index{correspondence!cohomological}
from $(X,M)$ to $(Y,N)$ is a triple of the form
$(C,c,\alpha)$, where $(C,c)$ determines the commutative diagram
\[
\begin{tikzcd}
C\ar[bend left]{rrd}{c_{2}}\ar[bend right]{ddr}[swap]{c_{1}}\ar{rd}{c}&&\\
 & X\times Y\ar{r}{p_{Y}}\ar{d}{p_{X}} & Y\\
 & X
\end{tikzcd}
\]
together with
a map $\alpha:c_{1}^{*}M\to c_{2}^{!}N$ in $\DM_{h}(C,\Lambda)$. 
\end{defn}
\begin{rem}
We have:
\[
\uHom(c_1^*M,c_2^!N)\cong
\uHom(c^{*}p_{X}^{*}M,c^{!}p_{Y}^{!}N)\cong c^{!}\uHom(p_{X}^{*}M,p_{Y}^{!}N)
\cong c^{!}(\mathbf{D}M\boxtimes N).
\]
Therefore, one can see $\alpha$ as a map of the form
\[\alpha:\Lambda\to c^!(\mathbf{D}M\boxtimes N)\, .\]
\end{rem}
\begin{rem}
In the case where $c_2$ is proper, a cohomological correspondence
induces a morphism in cohomology as follows. Let $a:X\to Spec(k)$
and $b:Y\to Spec(k)$ be the structural maps. We e have
$a c_1=bc_2$ and a co-unit map $(c_2)_*c^!_2(N)\to N$, whence
a map:
\[a_*M\to a_*(c_1)_*c^*_1M\xrightarrow{a_*(c_1)_*\alpha}
a_*(c_1)_*c^!_2N\cong b_*(c_2)_*c^!_2N\to b_*N\, .\]
In particular, one can consider the trace of such an induced
map. By duality, in the case where $c_1$ is proper, we get an induced map
in cohomology with compact support $b_!N\to a_!M$.
\end{rem}
\begin{paragr}
We observe that cohomological correspondences can be multiplied:
given another cohomological correspondence $(C',c',\alpha')$
from $(X',M')$ to $(Y',N')$, we define
a new correspondence from $(X\times X',M\boxtimes M')$
to $(Y\times Y',N\boxtimes N')$ with
\[(C,c,\alpha)\otimes(C',c',\alpha')=
(C\times C',c\times c',\alpha\boxtimes\alpha')\]
where $\alpha\boxtimes\alpha'$ is defined using the functoriality of the
$\boxtimes$ operation together with the canonical K\"unneth isomorphisms seen in the
previous paragraph:
\[\Lambda\cong\Lambda\boxtimes\Lambda
\xrightarrow{\alpha\boxtimes\alpha'}
c^!(\mathbf{D}M\boxtimes N)\boxtimes c^{\prime\, !}(\mathbf{D}M'\boxtimes N')
\cong(c\times c')^!(\mathbf{D}(M\boxtimes M')\boxtimes (N\boxtimes N'))\, .\]
\end{paragr}
\begin{paragr}
Correspondences can also be composed. Let $(C,c,\alpha)$ be a correspondence
from $(X,M)$ to $(Y,N)$ as above,
and let $(D,d,\beta)$ be a correspondence from $(Y,N)$
to $(Z,P)$, with $(D,d)$ corresponding
to a commutative diagram of the form below,
and $\beta:\Lambda\to d^!(\mathbf{D}N\boxtimes P)$
a map in in $\DM_{h}(D,\Lambda)$.
\[
\begin{tikzcd}
D\ar[bend left]{rrd}{d_{2}}\ar[bend right]{ddr}[swap]{d_{1}}\ar{rd}{d}&&\\
 & Y\times Z\ar{r}{p_{Z}}\ar{d}{p_{Y}} & Z\\
 & Y
\end{tikzcd}
\]
We form the following pullback square
\[\begin{tikzcd}
E\ar{r}{\lambda}\ar{d}{\mu}&D\ar{d}{d_1}\\
C\ar{r}{c_2}&Y
\end{tikzcd}\]
as well as the commutative diagram
\[
\begin{tikzcd}
E\ar[bend left]{rrd}{e_{2}}\ar[bend right]{ddr}[swap]{e_{1}}\ar{rd}{e}&&\\
 & X\times Z\ar{r}{p_{Z}}\ar{d}{p_{X}} & Z\\
 & X
\end{tikzcd}
\]
in which $e_1=c_1\mu$ and $e_2=d_2\lambda$.
We then form $\alpha\boxtimes\beta$:
\[\Lambda\cong\Lambda\boxtimes\Lambda\xrightarrow{\alpha\boxtimes\beta}
c^!(\mathbf{D}M\boxtimes N)\boxtimes d^!(\mathbf{D}N\boxtimes P)
\cong
(c\times d)^!((\mathbf{D}M\boxtimes N)\boxtimes(\mathbf{D}N\boxtimes P))\, .\]
Let $f=d_1\lambda=c_2\mu:E\to Y$ be the canonical map, and
$\Delta:Y\to Y\times Y$ be the diagonal. We have the following
Cartesian square
\[
\begin{tikzcd}
E\ar{r}{(\mu,\lambda)}\ar{d}[swap]{\varphi=(e_1,f,e_2)}
&C\times D\ar{d}{c\times d}\\
X\times Y\times Z\ar{r}{1\times\Delta\times 1}&X\times Y\times Y\times Z
\end{tikzcd}
\]
which induces an isomorphism (proper base change formula)
\[\varphi_!(\mu,\lambda)^*
\cong(1\times\Delta\times 1)^*(c\times d)_!\, .\]
In particular, it induces a canonical map
\[\kappa:(\mu,\lambda)^*(c\times d)^!\to \varphi^!(1\times\Delta\times 1)^*\]
corresponding by adjunction to the composite
\[\varphi_!(\mu,\lambda)^*(c\times d)^!
\cong(1\times\Delta\times 1)^*(c\times d)_!(c\times d)^!
\xrightarrow{\text{co-unit}} (1\times\Delta\times 1)^*\, .\]
Let $\pi:X\times Y\times Z\to X\times Z$ be the canonical
projection. There is a canonical map
\[\varepsilon:(1\times\Delta\times 1)^*
(\mathbf{D}M\boxtimes N\boxtimes\mathbf{D}N\boxtimes P)
\to\pi^!(\mathbf{D}M\boxtimes P)\]
induced by the evaluation map
\[N\otimes\mathbf{D}N\to I_Y\]
together with the canonical identifications coming
from appropriate K\"unneth formulas:
\begin{align*}
(1\times\Delta\times 1)^*
(\mathbf{D}M\boxtimes (N\boxtimes\mathbf{D}N)\boxtimes P)
&\cong \mathbf{D}M\boxtimes(N\otimes\mathbf{D}N)\boxtimes P\\
\mathbf{D}M\boxtimes I_Y\boxtimes P
&\cong \pi^!(\mathbf{D}M\boxtimes P)\, .
\end{align*}
We observe that $e=\pi\varphi$, so that $e^!\cong\varphi^!\pi^!$.
\end{paragr}
\begin{defn}
With the notations above, composing $(\mu,\lambda)^*(\alpha\boxtimes\beta)$
with the maps $\kappa$ and $\varepsilon$ defines
the map
\[
\beta\circ\alpha:\Lambda\cong(\mu,\lambda)^*\Lambda
\to\varphi^!\pi^!(\mathbf{D}M\boxtimes P)
\cong e^!(\mathbf{D}M\boxtimes P)\, .
\]
We define finally define the
\emph{composition of the correspondences}
$(C,c,\alpha)$ and $(D,d,\beta)$
as\label{compose corr}
\[(D,d,\beta)\circ(C,c,\alpha)=(E,e,\beta\circ\alpha)\, .\]
\end{defn}
\begin{paragr}
This composition is only well defined up to isomorphism
(since some choice of pull-back appears), but it is
associative and unital up to isomorphism. The unit
cohomological correspondence of $(X,M)$ is given by
\[1_{(X,M)}=(X,\Delta,1_M)\]
where $\Delta:X\to X\times X$ is the diagonal map and
\[1_M:\Lambda\to\Delta^!(\mathbf{D}M\boxtimes M)
\cong\uHom(M,M)\]
is the canonical unit map.
In a suitable sense, this defines a symmetric monoidal
bicategory, where the tensor product is defined as
\[(X,M)\otimes(Y,N)=(X\times Y,M\boxtimes N)\]
while the unit object if $(Spec(k),\Lambda)$.
\end{paragr}
\begin{paragr}
To make this a little bit more precise, we must speak of
the category of cohomological correspondences
from $(X,M)$ to  $(Y,N)$, in order to be able to express
the fact that all the contructions and all the
coherence isomorphisms (expressing the associativity and so on)
are functorial. If $(C,c,\alpha)$ and $(D,d,\beta)$ both
are correspondences from $(X,M)$ to  $(Y,N)$, a map
\[\sigma:(C,c,\alpha)\to(D,d,\beta)\]
is a pair $\sigma=(f,h)$, where $f:C\to D$ is
a proper morphism such that $df=c$, while $h$ is a homotopy
\[h:f_!(\alpha)\cong \beta\]
where $f_!(\alpha)$ is the map defined as
\[f_!(\alpha):\Lambda\xrightarrow{\text{unit}} f_*\Lambda\xrightarrow{f_*\alpha}
f_*c^!(\mathbf{D}M\boxtimes N)
\cong f_*f^!d^!(\mathbf{D}M\boxtimes N)
\xrightarrow{\text{co-unit}} d^!(\mathbf{D}M\boxtimes N) .\]
This defines the symmetric monoidal bicategory $\corr(k)$
whose objects are the pairs $(X,M)$ formed of a $k$-scheme $X$
equipped with a $\Lambda$-linear locally constructible $h$-motive $M$.
In particular, for each pair of pairs $(X,M)$ and $(Y,N)$,
there is the category of
cohomological correspondences from $(X,M)$ to $(Y,N)$,
denoted by $\corr(X,M;Y,N)$
(in this paragraph, unless we make it explicit
otherwise, we will only need the $1$-category of such
things, considering maps $\alpha$ as above in the homotopy
category of $h$-motives).
\end{paragr}
\begin{prop}
All the objects of $\corr(k)$ are dualizable.
Moreover, the dual of a pair $(X,M)$ is $(X,\mathbf{D}M)$.\label{cohdual}
\end{prop}
\begin{proof}
Let $(X,M)$, $(Y,N)$ and $(Z,P)$ be three objects of
$\corr(k)$. A cohomological correspondence
from $(X\times Y,M\boxtimes N)$ to $(Z,P)$
is determined by a morphism of $k$-schemes
$c:C\to X\times Y\times Z$ together with a map
\[\alpha:\Lambda\to c^!(\mathbf{D}(M\boxtimes N)\boxtimes P)\, .\]
A cohomological correspondence
from $(X,M)$ to $(Y\times Z,\mathbf{D}N\boxtimes P)$
is determined by a morphism of $k$-schemes
$c:C\to X\times Y\times Z$ together with a map
\[\alpha:\Lambda\to
c^!(\mathbf{D}M\boxtimes (\mathbf{D}N\boxtimes P))\, .\]
The K\"unneth formula
\[\mathbf{D}(M\boxtimes N)\boxtimes P
\cong\mathbf{D}M\boxtimes (\mathbf{D}N\boxtimes P)\]
implies our assertion.
\end{proof}
\begin{paragr}
Let $X$ be a scheme and $M$ a locally constructible $h$-motive on $X$.
We denote by $\Delta:X\to X\times X$ the diagonal map.
There is a \emph{transposed evaluation map}
\[ev^t_M:\mathbf{D}M\boxtimes M\to\Delta_*I_X\]
which corresponds by adjunction to the classical evaluation map
\[\Delta^*(\mathbf{D}M\boxtimes M)\cong\uHom(M,I_X)\otimes M\to I_X\, .\]
\end{paragr}
\begin{defn}
Let $(C,c,\alpha)$ be a cohomological correspondence
from $(X,M)$ to $(Y,N)$.
In the case $(X,M)=(Y,N)$ 
we can form the following Cartesian square.
\[
\begin{tikzcd}
F\ar{r}{\delta}\ar{d}[swap]{p} & C\ar{d}{c}\\
X\ar{r}{\Delta} & X\times X
\end{tikzcd}
\]
The scheme $F$ is called the \emph{fixed locus}\index{fixed locus (of a correspondence)}
of the
correspondence $(C,c)$.
The transposed evaluation map of $M$ induces  by proper base change a map
\[c^!(ev^t_M):
c^!(\mathbf{D}M\boxtimes M)\to c^!\Delta_*I_X\cong\delta_*p^!I_X
\cong\delta_*I_F\, ,\]
and thus, by adjunction, a map
\[ev^t_{M,c}:\delta^*c^!(\mathbf{D}M\boxtimes M)\to I_F\, .\]
The map $\alpha:\Lambda\to c^!(\mathbf{D}M\boxtimes M)$ finally
induces a map
\[Tr(\alpha):\Lambda\cong p^*\Lambda\to I_{F}\]
defined as the composition of $\delta^*\alpha$
with $ev^t_{M,c}$ (modulo the identification $\delta^*\Lambda\cong\Lambda$).
The corresponding class
\[Tr(\alpha)\in H^0\Hom_{\DM_h(F,\Lambda)}(\Lambda,I_F)\]
is called the \emph{characteristic class}\index{characteristic class} of $\alpha$.
\end{defn}
\begin{Example}
Let $f:X\to X$ be a morphism of schemes, and let $M$
be a $\Lambda$-linear locally constructible $h$-motive on $X$,
equipped with a map $\alpha: f^*M\to M$.
Then $(X,(1_X,f),\mathbf{D}\alpha)$ is a cohomological correspondence
from $(X,\mathbf{D}M)$ to itself, with
\[\mathbf{D}\alpha: 1^*_X\mathbf{D}M\cong\mathbf{D}M\to\mathbf{D}f^*M\cong f^!\mathbf{D}M\, .\]
If we form the Cartesian square
\[
\begin{tikzcd}
F\ar{r}{\delta}\ar{d}[swap]{p} & X\ar{d}{(1_X,f)}\\
X\ar{r}{\Delta} & X\times X
\end{tikzcd}
\]
we see that $F$ is indeed the fixed locus of the morphism $f$.
If $\Lambda\subset\QQ$, then he associated characteristic class
\[Tr(\mathbf{D}\alpha)\in H^0\Hom(\Lambda,I_F)
\otimes\QQ\cong{CH}_0(F)\otimes\QQ\]
defines a $0$-cycle on $F$ (see Theorem \ref{cycle class}).
In the case where $f$ only has isolated
fixed points, we have
\[{CH}_0(F)\otimes\QQ\cong
{CH}_0(F_{red})\otimes\QQ
\cong\oplus_{i\in I}{CH}_0(Spec(k_i))\otimes\QQ\]
where $I$ is a finite set and each $k_i$ is a finite field extension of $k$
with $F_{red}=\coprod_i Spec(k_i)$.
Using this decomposition, one can then express the characteristic class of $\alpha$
as a sum of local terms: the contributions of each summand
${CH}_0(Spec(k_i))\otimes\QQ$.
For instance, if $U$ is an open subset of $X$ such that
$f(U)\subset U$, and if $j:U\to X$ is the  inclusion map,
we can consider $M=j_!\Lambda$ and
the canonical isomorphism $\alpha:f^*j_!\Lambda\to j_!\Lambda$,
in which case $Tr(\alpha)$ is a way to count the number of fixed points
of $f$ in $U$ with `arithmetic multiplicities'
(in the form of $0$-cycles).\label{example fixed points}
\end{Example}
\begin{rem}
The notation $Tr(\alpha)$ is justified by Proposition \ref{cohdual}:
indeed, essentially by definition of the
composition law for cohomological correspondences sketched
in paragraph \ref{compose corr}, the characteristic class $Tr(\alpha)$
is the trace of the endomorphism $(C,c,\alpha)$ of the dualizable
object $(X,M)$. Indeed, the endomorphisms of $(Spec(k),\Lambda)$
in $\corr(k)$ are determined by pairs $(F,t)$
where $F$ is a $k$-scheme and $t:\Lambda\to I_F$ is
a section of the dualizing object of $F$ in $\DM_h(F,\Lambda)$.
\end{rem}
\begin{cor}
For any cohomological correspondences $(C,c,\alpha)$
and $(D,d,\beta)$ from $(X,M)$ to itself, we have:
\[Tr(\beta\circ\alpha)=Tr(\alpha\circ\beta)\, .\]
\end{cor}
\begin{cor}
Let $(C,c,\alpha)$ be a cohomological correspondence
from $(X,M)$ to itself. If we see $\alpha$ as a map
from $c_1^*M\to c^!_2M$, it determines a map
\[\mathbf{D}\alpha:c^*_2\mathbf{D}M\cong\mathbf{D}c_2^!M\to
\mathbf{D}c_1^*M\cong c^!_1\mathbf{D}M\, .\]
If $\tau:X\times X\to X\times X$ denotes the permutation
of factors, the cohomological correspondence $(C,\tau c,\mathbf{D}\alpha)$
from $(X,\mathbf{D}M)$ to itself is
the explicit description of the map obtained from
$(C,c,\alpha)$ by duality. In particular:
\[Tr(\alpha)=Tr(\mathbf{D}\alpha)\, .\]
\end{cor}
\begin{paragr}
The formation of traces is functorial with respect to
morphisms of correspondences. Let
$M$ be a locally constructible motive on a scheme $X$,
and $f:C\to D$, $d:D\to X\times X$, $c=df$, be morphisms,
with $f$ proper. We form pull-back squares
\[
\begin{tikzcd}
F\ar{r}{\delta}\ar{d}{g}
\ar[bend right]{dd}[swap]{p}&C\ar{d}[swap]{f}\ar[bend left]{dd}{c}\\
G\ar{r}{\varepsilon}\ar{d}{q}&D\ar{d}[swap]{d}\\
X\ar{r}{\Delta}&X\times X
\end{tikzcd}
\]
and have a composition
\[f_*c^!(\mathbf{D}M\boxtimes N)
\cong f_*f^!d^!(\mathbf{D}M\boxtimes N)
\xrightarrow{\text{co-unit}} d^!(\mathbf{D}M\boxtimes N)\]
as well as a composition
\[f_*\delta_* I_F\cong\varepsilon_* g_* I_F
\cong\varepsilon_* g_* g^! I_G
\xrightarrow{\text{co-unit}}\varepsilon_* I_G\, .\]
One then checks right away that the following square commutes.
\[
\begin{tikzcd}
f_*c^!(\mathbf{D}M\boxtimes N)\ar{r}\ar{d}[swap]{f_*c^!(ev^t_M)}
&d^!(\mathbf{D}M\boxtimes N)\ar{d}{g_*d^!(ev^t_M)}\\
f_*\delta_* I_F\ar{r}&\varepsilon_*I_G
\end{tikzcd}
\]
This implies immediately that, for any map\label{trivial functoriality trace}
$\alpha:\Lambda\to c^!(\mathbf{D}M\boxtimes M)$, we have:
\[Tr(\alpha)=Tr(f_!(\alpha))\, .\]
\end{paragr}
\begin{paragr}
Proper maps act on cohomological correspondences
as follows. We consider a proper
morphism of geometric correspondences, by which we mean a
commutative square of the form
\[\begin{tikzcd}
C\ar{r}{\varphi}\ar{d}[swap]{c=(c_1,c_2)}&D\ar{d}{d=(d_1,d_2)}\\
X\times X'\ar{r}{f\times f'}&Y\times Y'
\end{tikzcd}\]
in which $f:X\to Y$, $f':X'\to Y'$ and $\varphi:C\to D$ are proper
map, together with locally constructible $h$-motives
$M$ on $X$ and $M'$ on $X'$. Given a cohomological
correspondence from $(X,M)$ to $(X',M')$ of the form $(C,c,\alpha)$,
we have a cohomological correspondence from $(X,f_!M)$ to $(X',f'_!M')$
\[(f,f')_!(C,c,\alpha)=(C,d\varphi,(f,f')_!(\alpha))\]
defined as follows. If, furthermore, the commutative square
above is Cartesian, the map $(f,f')_!(\alpha)$ is the induced
map
\[\Lambda\xrightarrow{\text{unit}}
\varphi_*\Lambda\xrightarrow{\varphi_*\alpha}
\varphi_* c^!(\mathbf{D}M\boxtimes M')
\cong d^!(f\times f')_*(\mathbf{D}M\boxtimes M')
\cong d^!(\mathbf{D}f_!M\boxtimes f'_!M')\]
Otherwise, we consider the induced proper map
\[g:C\to E=X\times X'\times_{Y\times Y'}D\]
and apply the preceding construction to $g_!(\alpha)$, replacing
$C$ by $E$. 

In the case where $(X,M)=(X',M')$ and $f=f'$, we simply write\label{pushcorr}
\[f_!(\alpha)=(f,f)_!(\alpha)\, .\]
\end{paragr}
\begin{thm}[Lefschetz-Verdier Formula]\index{trace formula!Lefschetz-Verdier}
We consider a commutative square of $k$-schemes of finite type
of the form
\[\begin{tikzcd}
C\ar{r}{\varphi}\ar{d}[swap]{c=(c_1,c_2)}&D\ar{d}{d=(d_1,d_2)}\\
X\times X\ar{r}{f\times f}&Y\times Y
\end{tikzcd}\]
in which both $f$ and $\varphi$ are proper, as well as a locally
constructible $h$-motive $M$ on $X$, together with a map
$\alpha:\Lambda\to c^!(\mathbf{D}M\boxtimes M)$.
Let $F$ and $G$ be the fixed locus of $(C,c)$ and $(D,d)$ respectively.
Then the induced map $\psi:F\to G$ is also proper, and\label{thm:Lefschetz-Verdier}
\[\psi_!(Tr(\alpha))=Tr(f_!(\alpha))\, .\]
\end{thm}
\begin{proof}
The functoriality of the trace explained in \ref{trivial functoriality trace}
shows that it is sufficient to prove the theorem in the case where the
square is Cartesian.
We check that the two maps
\[(f\times f)_*(\mathbf{D}M\boxtimes M)
\cong (\mathbf{D}f_!M\boxtimes f_!M)\xrightarrow{ev^t_{f_!M}}
\Delta_* I_Y\]
and
\[(f\times f)_*(\mathbf{D}M\boxtimes M)
\xrightarrow{(f\times f)_*(ev^t_{M})}(f\times f)_*\Delta_* I_X
\cong\Delta_* f_* I_X\cong\Delta_*f_! f^!I_Y
\xrightarrow{\text{co-unit}}\Delta_*I_Y\]
are equal (where we have denoted by the same symbol the
diagonal of $X$ and the diagonal of $Y$).
By duality, this amounts to check that the unit map
\[\Delta_*\Lambda\to M\boxtimes\mathbf{D}M\]
is compatible with the push-forward $f_*$.
This is a fancy way to say that $f_*M$ has
a natural $f_*\Lambda$-algebra structure, which comes
from the fact that the functor $f^*$ is symmetric monoidal.
The Lefschetz-Verdier Formula follows then right away.
\end{proof}
\begin{rem}
When $\Lambda=\QQ$, the operator $\psi_!$ coincides
with the usual push-forward of $0$-cycles:
seen as a map
\[\psi_!:H^0\Hom(\Lambda,I_F)\to H^0\Hom(\Lambda,I_G)\, . \]
\end{rem}
\begin{thm}[Additivity of Traces]\index{trace formula!additivity}
Let $c=(c_1,c_2):C\to X\times X$ be a correspondence of $k$-schemes.
We consider a cofiber sequence
\[M'\to M\to M''\]
in $\DM_{h,lc}(X)$ as well as maps
\[\alpha' : c_1^*M' \to c_2^!M'\, , \
\alpha : c_1^*M \to c_2^!M \, , \
\alpha'' : c_1^*M'' \to c_2^!M''\]
in $\DM_{h,lc}(C)$ so that the diagram below commutes (in the sense of $\infty$-categories).
\[
\begin{tikzcd}
c^*_1M'\ar{r}\ar{d}{\alpha'}
&c^*_1M\ar{r}\ar{d}{\alpha}
&c^*_1M''\ar{d}{\alpha''}\\
c^!_2M'\ar{r}&c^!_2M\ar{r}&c^!_2M''
\end{tikzcd}
\]
Then the following formula holds.
\[Tr(\alpha)=Tr(\alpha')+Tr(\alpha'')\]
\end{thm}
The proof is given in the paper of Jin and Yang
\cite[Theorem 4.2.8]{trace} using the language
of algebraic derivators, which is sufficient for our purpose
(note however that, by Balzin's work \cite[Theorem 2]{balzin},
it is clear that one can go back and forth between the language of fibred
$\infty$-categories and the one of algebraic derivators).
The additivity of traces can be extended to more general homotopy colimits;
see Gallauer's thesis~\cite{Gallauer}.
\begin{rem}
It is pleasant to observe that, when $\Lambda=\QQ$, this
is the classical push-forward of $0$-cycles.
Lefschetz-Verdier Formula is particularly relevant
in the case where
$Y=Spec(k)$, and $F$ consists of isolated points in $X$
(in which case it is called the Grothendieck-Lefschetz formula).
Indeed, $f_!(M)$ is then the cohomology
of $X$ with compact support with coefficients in $M$,
so that $Tr(f_!(\alpha))$ is the ordinary trace of
the endomorphism $f_!(\alpha):f_!(M)\to f_!(M)$,
which can be computed through $\ell$-adic realizations
as an alternating sum of ordinary traces of linear maps. On the other hand,
$\psi_!(Tr(\alpha))$ is the sum of traces
of the endomorphisms induced by $\alpha$
on each $p_!x^*M$, where $x$ runs over the points of $F$,
with $p:Spec(\kappa(x))\to Spec(k)$ the structural map.
In the particular case discussed at the end of
Example \ref{example fixed points},
this shows that one can compute the number of fixed points
with geometric multiplicities of a endomorphism
of a $k$-scheme $f:X\to X$ with
isolated fixed points which extends to an endomorphism
of a compactification of $X$ and whose graph is
transverse to the diagonal, using the trace of
the induced endomorphism of the motive with compact support of $X$.
For the Frobenius map, such an extension is automatic, so that
We can count rational points of any
separated $\mathbf{F}_q$-scheme of finite type $X_0$
over a finite field $\mathbf{F}_q$ with the
Grothendieck-Lefschetz formula
\[ \# X(\mathbf{F}_q)=\sum_i (-1)^i\, \mathrm{Tr}\big(F:H^i_c(X,\QQ_\ell)\to H^i_c(X,\QQ_\ell)\big)\, ,\]
where $X$ is the pull-back of $X_0$ on the algebraic closure $\mathbf{F}_q$,
and where $F$ is a the map induced by the geometric Frobenius
(i.e. where one considers the correspondence defined by the
transposed graph of the arithmetic Frobenius).
Indeed, using
the additivity of traces, it is in fact sufficient to prove
this formula in the case where $X$ is smooth and projective, in which
case the classical Lefschetz formula applies.
We will now prove a more general version of it:
we will consider arbitrary (locally) constructible motivic
sheaves as coefficients.\label{rem:baby Lefschetz}
\end{rem}
\begin{paragr}
Let $p$ be a prime number, $r>0$ a natural number, and $q=p^r$.
Let $k_0=\mathbf F_q$ be the finite field with $q$ elements,
and let us choose an algebraic closure $k$ of $k_0$.
Given a $\mathbf F_p$-scheme $X$, we denote by
\[
F_X:X\to X
\]
the \emph{absolute Frobenius}\index{Frobenius!absolute}
of $X$, given by the identity on the underlying
topological space, and by $a\mapsto a^p$ on the structural sheaf $\mathcal{O}_X$.
The absolute Frobenius is a natural transformation from the identity of the category
of $\mathbf F_p$-schemes to itself. In particular, for any morphism of $k$-schemes
$u:U\to X$, there is a commutative square
$$\begin{tikzcd}
U\ar{r}{F_U}\ar{d}{u}&U\ar{d}{u}\\
X\ar{r}{F_X}&X
\end{tikzcd}$$
and thus a comparison map:
\[
F_{U/X}=(u,F_U)\colon U\to F_X^{-1}(U)=X\times_X U
\]
called the \emph{relative Frobenius}\index{Frobenius!relative} of $U$ over $X$.
In case $X_0$ is a $k_0$-scheme, the $r$th iteration of the absolute Frobenius
$$F^r_{X_0}\colon X_0\to X_0$$
is often called the \emph{$q$-absolute
Frobenius of $X_0$}\index{Frobenius!q-absolute@$q$-absolute}
(and has the feature of being
a map of $k_0$-schemes). By base change to $k$, it induces
the \emph{geometric Frobenius of $X$}\index{Frobenius!geometric}, i.e. the morphism
of $k$-schemes
\[
\phi_r:X\to X\, ,
\]
where $X=\spec{k}\times_{\spec{k_0}}X_0$. Following Deligne's conventions,
sheaves (or motives) on $X_0$ will often be denoted by $M_0$, and the
pullback of $M_0$ along the canonical projection $X\to X_0$ will be written $M$.
The map $k\to k$, defined by $x\mapsto x^q$ is an automorphism
of $k_0$-algebras, which induces an isomorphism of $k_0$-schemes
\[
\mathit{Frob}_q\colon\spec{k}\to\spec{k}\, .
\]
It induces an isomorphism of $k_0$-schemes
$$\mathit{Frob}_{q,X}
=(\mathit{Frob}_q\times_{\spec{k_0}}1_{X_0})\colon X\to X$$
whose composition with $\phi_r$ is nothing else than the
absolute Frobenius of $X$. The map
$$\mathit{Frob}_{q,X}^{-1}\colon X\to X$$
is often called the \emph{arithmetic Frobenius of $X$}.

\end{paragr}
\begin{Lemma}
Let $X$ be a locally noetherian $\mathbf{F}_p$-scheme. The functor
\[
F^*_X:\DM_{h}(X,\Lambda)\to\DM_{h}(X,\Lambda)
\]
is the identity.
\end{Lemma}
\begin{proof}
Let $a:X\to\spec{\mathbf{F}_p}$ be the structural map.
We have a commutative diagram of the form
\[
\begin{tikzcd}
X\ar{r}{F_X}\ar{d}{a}&X\ar{d}{a}\\
\spec{\mathbf{F}_p}\ar[equals]{r}&\spec{\mathbf{F}_p}
\end{tikzcd}
\]
in which the map $F_X$ is
a universal homeomorphism (being integral, radicial and surjective) and thus invertible locally for the $\h$-topology.
In other words, the square above is Cartesian locally
for the $\h$-topology. By $\h$-descent, the functor
\[
F^*_X:\DM_{h}(X,\Lambda)\to\DM_{h}(X,\Lambda)
\]
thus acts as the identity.
\end{proof}
\begin{rem}
For a $k_0$-scheme $X_0$, since the composition of the geometric Frobenius
$\phi_r:X\to X$ with the inverse of the arithmetic Frobenius is the absolute
Frobenius, this shows that considering actions of the geometric
Frobenius or of the arithmetic Frobenius amount to the same thing, at least
as far as motivic sheaves are concerned. In fact the previous lemma is also a way to
define such actions.

Let $M_0$ be a motivic sheaf on $X_0$, i.e. an object of
$\DM_{h}(X_0,\Lambda)$.
Since $F_X=\mathit{Frob}_{q,X}\, \phi_r$ we have
\[
F^*_X(M)=M\simeq\phi^*_r\,\mathit{Frob}^*_{q,X}(M)\, .
\]
On the other hand, since $M_0$ is defined over $k_0$,
and $M=a^*(M_0)$, there is a canonical isomorphism
$$\mathit{Frob}^*_{q,X}(M)\cong M\, .$$
Therefore, we have a canonical isomorphism
$$\phi^*_r(M)\cong M=(1_X)^!(M)\, .$$
Since the locus of fixed points of $\phi_r$ is precisely
the (finite) set $X(k_0)$ of rational points of $X_0$
(seen as a discrete algebraic variety over $k$),
the Verdier trace of the isomorphism above defines a class
\[
L(M_0)=\mathit{Tr}\Big(\phi^*_r(M)\xrightarrow{\cong} (1_X)^!(M)\Big)
\in H_0(X(k_0),\Lambda)\, .
\]
If $p$ is invertible in $\Lambda$, we have simply
\[
H_0(X(k_0),\Lambda)\cong\Lambda^{X(k_0)}\, .
\]
Under this identification, the obvious function
\[\Lambda^{X(k_0)}\to\Lambda\ , \quad
f\mapsto \int f=\sum_{x\in X(k_0)} f(x)\]
coincides with the operator
\[
\psi_!:H_0(X(k_0),\Lambda)\to H_0(\spec{k},\Lambda)=\Lambda
\]
induced by the structural map $\psi:X(k_0)\to\spec{k}$.

The action of Frobenius is functorial: for any map $f:X_0\to Y_0$,
the induced action of $\phi^*_r$ on $f_!(M)$ via the
proper base change isomorphism
\[ \phi^*_r f_!\cong f_!\phi^*_r \]
coincides with the action defined as above in the
case of $f_!(M_0)$. There is a similar compatibility with the
canonical isomorphism $\phi_r^*f^*\cong f^*\phi^*_r$.

If $X_0$ is proper and if $j:U_0\to X_0$
is an open immersion, for any $M_0$ locally constructible
in $\DM_{h}(U_0,\QQ)$,
we thus get, as a special case of
the Lefschetz-Verdier Formula (Theorem \ref{thm:Lefschetz-Verdier})
\[\mathit{Tr}\Big(\phi^*_r:a_!(M)\to a_!(M)\Big)=\int L(j_!M_0)\in\QQ\]
where $a:U\to \spec k$ is the
structural morphism.\label{rem:construction of Frobenius actions}
\end{rem}
\begin{thm}[Grothendieck-Lefschetz Formula]
Let $j:U_0\to X_0$\index{trace formula!Grothendieck-Lefschetz}
is an open immersion into a proper scheme of finite type over a finite field $k_0$
and let $M_0$ be a locally constructible motivic sheaf in $\DM_{h}(U_0,\QQ)$.
For each rational point $x$ of $U_0$, we denote by $M_x$ the fiber of $M$
at the induced geometric point of $U$, on which there is a canonical
action of the geometric Frobenius (as a particular case of the construction
of Remark \ref{rem:construction of Frobenius actions}).
Then\label{thm:Grothendieck-Lefschetz}
\[\mathit{Tr}\Big(\phi^*_r:a_!(M)\to a_!(M)\Big)=
\sum_{x\in U(k_0)}\mathit{Tr}\Big(\phi^*_r:M_x\to M_x\Big)\, .\]
\end{thm}
\begin{proof}
The case where $M_0=\QQ$ is constant is well known
(see Remark \ref{rem:baby Lefschetz}).
This proves the case where $M_0=p_!(\QQ)$ for a map $p:Y_0\to U_0$.
The case of a direct factor of $p_!(\QQ)$ with $Y_0$ smooth and projective
can be proved in the same way: the projector defining our motive
is then given by some $\mathrm{dim}(Y)$-dimensional cycle $\alpha$
on $Y\times Y$ supported on $Y\times_X Y$ (see Theorem \ref{cycle class}).
We then observe that the Grothendieck-Lefschetz fixed point formula holds
(using proper base change formula and
Olsson's computation of local terms \cite[Prop. 5.5]{Olsson2}).
On the other hand, we see that the shift $[i]$ and the Tate twist $(n)$
are compatible with traces (they consist in multiplying by $(-1)^i$ and
by $1$, respectively).
By the additivity of traces,
we are comparing two numbers which only depend on the
class of $M_0$ in the Grothendieck group $K_0(\DM_{h,lc}(U_0,\QQ))$,
and it is sufficient to consider the case where $U_0=X_0$ is projective.
Using Bondarko's theory of motivic weights \cite[Prop. 3.3]{Bondarko},
we see that any class in $K_0(\DM_{h,lc}(U_0,\QQ))$ is a linear
combination of classes of motives which are
direct factors of $p_*(\QQ)(n)[i]$ for $n,i\in\ZZ$ and $p:Y_0\to U_0$
a projective morphism, with $Y_0$ smooth and projective.
This proves the formula in general.
\end{proof}
\begin{rem}
When $M_0=p_!(\QQ)$, with $p:Y_0\to U_0$ separated of finite type,
the Grothendieck-Lefschetz Formula expresses the trace of the Frobenius
action on cohomology with compact support of $Y$ as a sum of the
traces of the action of Frobenius on cohomology with compact support
of the fibers $Y_x$ of $Y$ over each $k_0$-rational point $x$ of $U$.
One can do similar constructions replacing the geometric Frobenius
action by any (functorially given) automorphism of $k$-schemes,
such as the identity.
The computation of the local terms given by the Lefschetz-Verdier Trace
Formula can then be
rather involved. For instance,
in the case of the identity (which means that we want
to compute Euler-Poincar\'e characteristic of cohomology with
compact support), the naive formula tends to fail (at least in
positive characteristic). The Grothendieck-Ogg-Shafarevitch
Formula is such a non-trivial computation in the case where $M_0$ is
dualizable on a smooth curve $U_0$: it measures the deffect of the
naive formula in terms of Swan conductors.
\end{rem}
\bibliographystyle{plain}
\bibliography{traces_LNM}
\printindex
\end{document}